\documentclass[a4paper,12pt,centertags,oneside]{amsart}
\usepackage{amsmath,amstext,amsthm,a4,amssymb,amscd}
\usepackage[mathscr]{eucal}
\usepackage{mathrsfs}
\usepackage{epsf}
\usepackage{typearea}
\usepackage{charter}

\usepackage{graphicx}
\usepackage[all]{xy}

\newcommand{\R}{\mathbb{R}}
\newcommand{\C}{\mathbb{C}}

\newcommand{\smooth}{{\mathscr{C}^\infty}}

\newcommand{\diff}{d}

\newcommand{\im}{\mathrm{Im}}

\newcommand{\n}{\nabla}
\newcommand{\trs}{\tr_\mathrm{s}}

\DeclareMathOperator{\tr}{Tr}

\newtheorem{prop}{Proposition}[section]
\newtheorem{thm}[prop]{Theorem}
\newtheorem{lemme}[prop]{Lemma}

\theoremstyle{definition}
\newtheorem{defn}[prop]{Definition}

\theoremstyle{remark}
\newtheorem{rem}[prop]{Remark}

\numberwithin{equation}{section}
\makeatletter
\newcommand*{\rom}[1]{\expandafter\@slowromancap\romannumeral #1@}
\makeatother

\author{Yeping ZHANG} 
\address{D{\'e}partement de Math{\'e}matiques, 
B{\^a}timent 425, 
Facult{\'e} des Sciences d'Orsay, 
Universit{\'e} Paris-Sud,  
F-91405 Orsay Cedex}
\email{yeping.zhang@math.u-psud.fr}

\title[The direct image of a flat fibration]{The direct image of \\ a flat fibration with complex fibers}

\begin{document}
\maketitle

\begin{abstract}
We consider a proper flat fibration with real base and complex fibers. 
First we construct odd characteristic classes for such fibrations by a method that generalizes constructions of Bismut-Lott \cite{bl}. 
Then we consider the direct image of a fiberwise holomorphic vector bundle, which is a flat vector bundle on the base.
We give a Riemann-Roch-Grothendieck theorem calculating the odd real characteristic classes of this flat vector bundle.
\end{abstract} 

\tableofcontents

\section{Introduction}
 
We consider a compact real manifold $X$ equipped with a flat vector bundle $F$. 
We equip $X$ with a Riemannian metric $g^{TX}$. 
We equip $F$ with a Hermitian metric $g^F$.
The real analytic torsion \cite{rs} is a spectral invariant 
of the Hodge Laplacian associated with the de Rham complex $\big(\Omega^\cdot(X,F),d^F\big)$. 
Let $\det H^\cdot(X,F)$ be the determinant line of the de Rham cohomology $H^\cdot(X,F)$. 
The Ray-Singer metric is the product of the real analytic torsion 
and of the $L^2$-metric on $\det H^\cdot(X,F)$.

Cheeger \cite{c-cm} and M{\"u}ller \cite{m-cm} proved independently that,  
if $g^F$ is flat, 
then the Ray-Singer metric is independent of $g^{TX}$. 
M{\"u}ller \cite{m-cm-2} also extended this result to the unimodular case, 
i.e., the induced metric on $\det F$ is flat.
In the general case, 
the dependence of the Ray-Singer metric on the metrics 
was calculated by Bismut-Zhang \cite{bz}.
They also established an extension of the Cheeger-M\"{u}ller theorem 
in the general case. 

Now let $\pi : M \rightarrow S$ be a real smooth fibration with compact fiber $X$.
Let $F$ be a complex flat vector bundle over $M$.
Bismut and Lott \cite{bl} established Riemann-Roch-Grothendieck formulas, 
which calculate the odd Chern classes of the direct image $R^\cdot\pi_{*}F$
in terms of the Euler class of the relative tangent bundle $TX$ 
and the corresponding odd Chern classes of $F$. 
When equipping the considered vector bundles with metrics, 
these classes can be represented by explicit differential forms.
By transgressing the equality of cohomology classes at the level of differential forms, 
they also obtained even analytic torsion forms on $S$, 
whose coboundary is equal to the difference between the differential forms 
appearing on the left and right hand side of the R.R.G. formula. 

In complex geometry, 
the objects parallel to the real analytic torsion and the Ray-Singer metric 
are known as the holomorphic analytic torsion and the Quillen metric, 
which were also introduced by Ray-Singer \cite{rs2}.
These notions were extended to holomorphic fibrations
by  Bismut-Gillet-Soul\'{e} \cite{bgs2} and  Bismut-K{\"o}hler \cite{bk}. 

In this paper, 
we consider a flat fibration $q:\mathcal{N}\rightarrow M$ with complex fiber $N$.
We equip $\mathcal{N}$ with a complex vector bundle $E$, 
which is holomorphic along $N$ and flat along horizontal directions in $\mathcal{N}$. 
Then $R^\cdot\pi_{*}E$ is a flat vector bundle over $M$.
We give a R.R.G. formula for the odd Chern 
classes of $R^\cdot\pi_{*}E$ in terms of the Todd class of the relative tangent bundle 
and of the  Chern classes of $E$.  
Moreover, by equipping the various vector bundles with Hermitian metrics, 
we construct even analytic torsion forms on $M$, 
which transgress the equality of the corresponding cohomology classes. 
The results contained in this paper were announced in \cite{yzhang}.

Let us now give more detail on the content of this paper.

\subsection{Chern-Weil theory and its extensions}

\

Let $M$ be a smooth manifold. 
Let $E$ be a complex vector bundle of rank $r$ over $M$.
Let $\n^E$ be a connection on $E$. 
Let $P$ be an invariant polynomial on $\mathfrak{gl}(r,\mathbb{C})$. 
Chern-Weil theory assigns a closed differential form
\begin{equation}
P(E,\n^E) \in \Omega^\mathrm{even}(M) \;.
\end{equation}
Its cohomology class $\big[P(E,\n^E)\big]\in H^\mathrm{even}(M)$ is independent of $\nabla^{E}$, 
and will be denoted by $P(E)$.
This theory will be referred to as the even Chern-Weil theory.
If $\n^E$ is a flat connection, i.e., $\n^{E,2}=0$, 
$P(E,\n^E)$ is just a constant function on $M$. 

A Chern-Weil theory for flat vector bundles was developed by Bismut-Lott \cite[\textsection 1]{bl}. 
Let $(E,\n^E)$ be a flat complex vector bundle over $M$.
Let $g^E$ be a Hermitian metric on $E$.
Let $f$ be an odd polynomial.
Bismut and Lott assigned a closed differential form
\begin{equation}
f(E,\n^E,g^E) \in \Omega^\mathrm{odd}(M) \;.
\end{equation}
Its cohomology class $\big[f(E,\n^E,g^E)\big]\in H^\mathrm{odd}(M)$ is independent of $g^E$, 
and will be denoted by $f(E,\n^E)$.
This theory will be referred to as the odd Chern-Weil theory. 

In this paper, we will construct characteristic classes for flat fibrations with complex fibers. 
Our construction is a mixture of the even and odd Chern-Weil theory. 

Now we construct our flat fibration. 
Let $G$ be a Lie group. 
Let $p : P_G\rightarrow M$ be a flat $G$-principal bundle. 
Let $N$ be a compact complex manifold. 
We assume that $G$ acts holomorphically on $N$. 
Set
\begin{equation}
\mathcal{N} = P_G \times_G N \;.
\end{equation}
Let 
\begin{equation}
q : \mathcal{N} \rightarrow M
\end{equation}
be the canonical projection. 
Then $q$ defines a flat fibration with canonical fiber $N$. 
Let $E_0$ be a holomorphic vector bundle over $N$.
We assume that the action of $G$ lifts holomorphically to $E_0$. 
Set
\begin{equation}
E = P_G \times_G E_0 \;.
\end{equation}
Then $E$ is a complex vector bundle over $\mathcal{N}$. 

In \textsection \ref{dilab-sec-charac}, 
we construct odd characteristic forms as follows. 
We denote
\begin{align}
\begin{split}
\Omega^\cdot(M) \; & = \smooth\big(M,\Lambda^\cdot(T^*M)\big) \;,\\
\Omega^\cdot(\mathcal{N},E) \; & = \smooth\big(\mathcal{N},\Lambda^\cdot(T^*\mathcal{N})\otimes E\big) \;.
\end{split}
\end{align}
Let $d_M$ be the de Rham operator on $\Omega^\cdot(M)$. 
Let $d^E_M$ be the lift of $d_M$ to $\Omega^\cdot(\mathcal{N},E)$. 
Let $g^E$  be a Hermitian metric on $E$.
Set
\begin{equation}
\omega^E = \big(g^E\big)^{-1} d^E_M\,g^E\in \Omega^1(\mathcal{N},\mathrm{End}(E)) \;.
\end{equation}
Let $\n^E_N$ be the fiberwise Chern connection on $(E,g^E)$. 
Let $A^E$ be the unitary connection on $E$ defined by
\begin{equation}
A^E = \n^E_N + d^E_M + \frac{1}{2}\omega^E \;.
\end{equation} 
Let $r$ be the rank of $E$. 
Let $\mathfrak{gl}(r,\mathbb{C})$ be the Lie algebra of $GL(r,\mathbb{C})$. 
Let $P$ be an invariant polynomial on $\mathfrak{gl}(r,\mathbb{C})$. 
For $a,b\in\mathfrak{gl}(r,\mathbb{C})$, we use the following notation
\begin{equation}
\label{dilab-eq-intro-def-Pder}
\Big\langle P'(a),b \Big\rangle = {\frac{\partial}{\partial t} P(a+tb)}_{t=0} \;.
\end{equation}
Let $N^{\Lambda^\cdot (T^*\mathcal{N})}$ be the number operator of $\Lambda^\cdot (T^*\mathcal{N})$,
i.e., 
for $\alpha\in\Lambda^k (T^*\mathcal{N})$,  
$N^{\Lambda^\cdot (T^*\mathcal{N})}\alpha = k\alpha$.
Put
\begin{align}
\label{dilab-eq-intro-P-tildeP}
\begin{split}
P(E,g^E) \; & = 
(2\pi i)^{-\frac{1}{2}N^{\Lambda^\cdot (T^*\mathcal{N})}} P(-A^{E,2})
\in \Omega^\mathrm{even}(\mathcal{N}) \;,\\
\widetilde{P}(E,g^E) \; & = 
(2\pi i)^{\frac{1}{2}-\frac{1}{2}N^{\Lambda^\cdot (T^*\mathcal{N})}} \Big\langle P'(-A^{E,2}),\frac{\omega^E}{2} \Big\rangle
\in \Omega^\mathrm{odd}(\mathcal{N}) \;.
\end{split}
\end{align}

\begin{thm}
\label{dilab-thm-intro-ch-class}
The differential form
\begin{equation}
q_*\big[P(E,g^E)\big] \in \Omega^\mathrm{even}(M) 
\end{equation}
is a constant function. 

The differential form
\begin{equation}
q_*\big[\widetilde{P}(E,g^E)\big] \in \Omega^\mathrm{odd}(M) 
\end{equation}
is closed. Its cohomology class 
\begin{equation}
\left[q_*\big[\widetilde{P}(E,g^E)\big]\right]\in H^\mathrm{odd}(M)
\end{equation}
is independent of $g^E$.  
\end{thm}

In the sequel, we use the notation
\begin{equation}
q_*\big[\widetilde{P}(E)\big] = \left[q_*\big[\widetilde{P}(E,g^E)\big]\right]\in H^\mathrm{odd}(M) \;.
\end{equation}

Now let $F$ be another vector bundle (of rank $r'$) over $\mathcal{N}$ satisfying the same properties as $E$. 
Let $g^F$ be a Hermitian metric on $F$. 
Let $Q$ be an invariant polynomial on $\mathfrak{gl}(r',\mathbb{C})$. 
The natural product on the forms $\widetilde{P}(E,g^E)$ and $\widetilde{Q}(F,g^F)$ is given by
\begin{equation}
\label{dilab-eq-intro-tildePQ}
\widetilde{P}(E,g^E) * \widetilde{Q}(F,g^F) = \widetilde{P}(E,g^E)Q(F,g^F) + P(E,g^E)\widetilde{Q}(F,g^F) \;.
\end{equation}

\subsection{A R.R.G. theorem for flat fibrations with complex fibers}

\label{dilab-subsect-intro-rrg}

\

In the rest of the introduction, we suppose that $N$ is a K{\"a}hler manifold.

Let $H^\cdot(N,E)$ be the fiberwise Dolbeault cohomology group of $E$ along $N$.
Then $H^\cdot(N,E)$ is a graded flat vector bundle over $M$. 
Let $\n^{H^\cdot(N,E)}$ be its flat connection. 

Set $f(x) = x\exp(x^2)$.
Let
\begin{equation}
f\big(H^\cdot(N,E),\n^{H^\cdot(N,E)}\big) \in H^\mathrm{odd}(M,\mathbb{R})
\end{equation}
be the Bismut-Lott odd characteristic class \cite[\textsection 1]{bl}.
 
\begin{thm}
\label{dilab-intro-thm-riemann-roch-grothendieck}
We have
\begin{equation}
f\big(H^\cdot(N,E), \nabla^{H^\cdot(N,E)}\big) \\
= q_*\big[\widetilde{\mathrm{Td}}(TN)*\widetilde{\mathrm{ch}}(E)\big]
\in H^\mathrm{odd}(M,\mathbb{R}) \;.
\end{equation}
\end{thm}

Here $\widetilde{\mathrm{Td}}(TN)*\widetilde{\mathrm{ch}}(E)$ 
is defined by \eqref{dilab-eq-intro-P-tildeP} and \eqref{dilab-eq-intro-tildePQ}.

Now we explain the idea of the proof. 
We will use the superconnection formalism \cite[\textsection 2]{bl}.
Put
\begin{equation}
\mathscr{E} = \mathscr{C}^\infty(N,\Lambda^\cdot (\overline{T^*N})\otimes E ) \;,
\end{equation}
which is an infinite dimensional flat vector bundle over $M$. 
Let $d^\mathscr{E}_M$ be its flat connection. 
Let $\overline{\partial}^E_N$ be the Dolbeault operator on $\mathscr{E}$. 
Set
\begin{equation}
A^\mathscr{E} = \overline{\partial}^E_N + d^\mathscr{E}_M \;,
\end{equation}
which acts on $\Omega^\cdot(M,\mathscr{E})$. 
Here $A^\mathscr{E}$ is a flat superconnection on $\mathscr{E}$ in the sense of Bismut-Lott \cite[Definition 1.1]{bl}.

Let $g^{TN}$ be a fiberwise K{\"a}hler metric on $TN$. 
Let $g^E$ be a Hermitian metric on $E$. 
Let $g^\mathscr{E}$ be the induced $L^2$-metric on $\mathscr{E}$. 
Let $A^{\mathscr{E},*}$ be the adjoint superconnection of $A^\mathscr{E}$ in the sense of Bismut-Lott \cite[Definition 1.6]{bl}.

Let $N^{\Lambda^\cdot(T^*M)}$ be the number operator of $\Lambda^\cdot(T^*M)$. 
Set
\begin{equation}
D^\mathscr{E} = 2^{-N^{\Lambda^\cdot(T^*M)}} \big(A^{\mathscr{E},*} - A^\mathscr{E}\big) 2^{N^{\Lambda^\cdot(T^*M)}} 
\in\Omega^\cdot(M,\mathrm{End}(\mathscr{E})) \;.
\end{equation} 
For $t>0$, let $D^\mathscr{E}_t$ be the operator $D^\mathscr{E}$ associated with the rescaled metric $\frac{1}{t}g^{TN}$.
Following Bismut-Lott \cite[(2.22),(2,23)]{bl}, we define
\begin{align}
\label{dilab-eq-intro-def-alpha-beta-t}
\begin{split}
\alpha_t & = (2\pi i)^{\frac{1}{2}-\frac{1}{2}N^{\Lambda^\cdot(T^*M)}} \trs\Big[D^\mathscr{E}_t\exp(D^{\mathscr{E},2}_t)\Big] \;,\\
\beta_t & = (2\pi i)^{-\frac{1}{2}N^{\Lambda^\cdot(T^*M)}} \trs\Big[ \frac{N^{\Lambda^\cdot(\overline{T^*N})}}{2}(1+2D^{\mathscr{E},2}_t) \exp(D^{\mathscr{E},2}_t) \Big] \;.
\end{split}
\end{align}
We have
\begin{equation}
\label{dilab-eq-intro-alpha-t-beta-t}
d_M \alpha_t = 0 \;,\hspace{5mm}
\frac{\partial}{\partial t} \alpha_t = \frac{1}{t}d_M \beta_t \;.
\end{equation}

Let $g^{H^\cdot(N,E)}$ be the metric on $H^\cdot(N,E)$ induced by the $L^2$-metric on $\mathscr{E}$ via the Hodge theorem.
Let
\begin{equation}
f\big(H^\cdot(N,E),\n^{H^\cdot(N,E)},g^{H^\cdot(N,E)}\big) \in \Omega^\mathrm{odd}(M)
\end{equation}
be the Bismut-Lott odd characteristic form \cite[Definition 1.7]{bl}.

Theorem \ref{dilab-intro-thm-riemann-roch-grothendieck} is a consequence of the following theorem. 
\begin{thm}
\label{dilab-label-intro-asymptotics-alpha-t}
As $t\rightarrow\infty$,
\begin{equation}
\label{dilab-intro-eq-asymptotics-alpha-t-1}
\alpha_t  = 
f\big(H^\cdot(N,E), \nabla^{H^\cdot(N,E)}, g^{H^\cdot(N,E)}\big) + \mathscr{O}\big(\frac{1}{\sqrt{t}}\big) \;.
\end{equation}
As $t\rightarrow 0$,
\begin{equation}
\label{dilab-intro-eq-asymptotics-alpha-t-2}
\alpha_t =  
q_*\big[\widetilde{\mathrm{Td}}(TN,g^{TN})*\widetilde{\mathrm{ch}}(E,g^E)\big] + \frac{\text{a fixed exact form }}{t} + \mathscr{O}\big(\sqrt{t}\big) \;.
\end{equation}
\end{thm}

\subsection{Analytic torsion forms}
\label{dilab-subsect-intro-torsion-form}

\

In the same way as in  \eqref{dilab-intro-eq-asymptotics-alpha-t-1} and \eqref{dilab-intro-eq-asymptotics-alpha-t-2}  , 
we also obtain an asymptotic estimate for $\beta_t$ as $t\rightarrow\infty$ and $t\rightarrow 0$. 
We construct explicitly an analytic torsion form
\begin{equation}
\mathscr{T}(g^{TN},g^E) \in \Omega^\mathrm{even}(M) \;,
\end{equation}
which is defined by subtracting the singularities of the following integral
\begin{equation}
- \int_0^\infty \beta_t \frac{dt}{t} \;.
\end{equation}
By \eqref{dilab-eq-intro-alpha-t-beta-t}, \eqref{dilab-intro-eq-asymptotics-alpha-t-1} and \eqref{dilab-intro-eq-asymptotics-alpha-t-2},
we have
\begin{align}
\begin{split}
& d_M \mathscr{T}(g^{TN},g^E) \\
= \; & q_*\big[\widetilde{\mathrm{Td}}(TN,g^{TN})*\widetilde{\mathrm{ch}}(E,g^E)\big] 
- f\big(H^\cdot(N,E), \n^{H^\cdot(N,E)}, g^{H^\cdot(N,E)}\big) \;.
\end{split}
\end{align}
Moreover, we show that the degree zero component of $\mathscr{T}(g^{TN},g^E) $ is 
the Ray-Singer holomorphic torsion \cite{rs2} associated with $(N,g^{TN},E,g^E)$.

This paper is organized as follows.

In \textsection \ref{dilab-sect-pre}, 
we recall several standard constructions and known results. 
Most of them can be found in \cite{bgv} and \cite[\textsection 1]{bl}.

In \textsection \ref{dilab-sec-charac}, 
we construct characteristic classes for flat fibrations
and prove Theorem \ref{dilab-thm-intro-ch-class}.

In \textsection \ref{dilab-section-rrg},
we prove Theorem \ref{dilab-label-intro-asymptotics-alpha-t}.
As a consequence, we establish Theorem \ref{dilab-intro-thm-riemann-roch-grothendieck}.
We also construct the analytic torsion form $\mathscr{T}(g^{TN},g^E)$.

\

\textbf{Acknowledgment}

This paper is part of the author's PhD thesis.
The author would like to thank his advisor Professor Jean-Michel Bismut for his guidance. 
The research leading to the results contained in this paper has received funding from
the European Research Council (E.R.C.) under European Union's Seventh Framework Program 
(FP7/2007-2013)/ ERC grant agreement No. 291060.

\section{Preliminaries}
\label{dilab-sect-pre}

\

This section is organized as follows. 

In \textsection \ref{dilab-subsec-superalg}, 
we introduce the superalgebra formalism.

In \textsection \ref{dilab-subsect-clifford-pre}, 
we introduce the Clifford algebra.

In \textsection \ref{dilab-subsec-evenodd-char}, 
we introduce the Chern-Weil theory.

In \textsection \ref{dilab-subsec-fibration-connection-metric},
we introduce several objects associated with a smooth fibration.

The constructions and results contained in this section can be found in 
\cite[\textsection 1]{bgv},
\cite[\textsection 1]{b}, 
\cite[\textsection 1]{bl}.

\subsection{Superalgebras}
\label{dilab-subsec-superalg}

\

In the sequel, all the algebras  will be over $\mathbb{R}$ or $\mathbb{C}$.

\begin{defn}
A superalgebra is 
an algebra $A$ equipped with a $\mathbb{Z}_2$-grading $A = A^{+} 
\oplus A^{-}$
such that
\begin{equation}
A^+A^\pm \subseteq A^\pm \;,\hspace{5mm}
A^-A^\pm \subseteq A^\mp \;.
\end{equation}
\end{defn}

Let $A$ be a superalgebra. 
An element $a\in A$ is said to be homogeneous if $a\in A^\pm$. 
We denote $\deg a = 0$ (resp.  $\deg a = 1$) 
if $a\in A^+$ (resp. $a\in A^-$).

The supercommutator of two homogeneous elements $a,b\in A$ is defined by
\begin{equation}
[a,b] = ab - (-1)^{\deg a \deg b}ba \;.
\end{equation}
Also $[\cdot,\cdot]$ extends by linearity to the whole superalgebra $A$. 

\begin{defn}
Let $A$ and $B$ be two superalgebras. 
The $\mathbb{Z}_2$-graded tensor product $A\widehat{\otimes}B$ is identified with $A\otimes B$ as vector spaces, and the multiplication is given by
\begin{equation}
(a_1\otimes b_2)\cdot(a_2\otimes b_2) = (-1)^{\deg a_2 \deg b_1} a_1a_2 \otimes b_1b_2 \;.
\end{equation}
\end{defn}

\begin{defn}
Let $A$ be a superalgebra. 
A super $A$-module is a $\mathbb{Z}_2$-graded vector space $V = V^+ \oplus V^-$
equipped with an action of $A$ such that
\begin{equation}
A^+ V^\pm \subseteq V^\pm \;,\hspace{5mm}
A^- V^\pm \subseteq V^\mp \;.
\end{equation}
\end{defn}

Let $V = V^+ \oplus V^-$ be a $\mathbb{Z}_2$-graded vector space. 
Set 
\begin{equation}
\tau = \mathrm{id}_{V^+} - \mathrm{id}_{V^-} \in \mathrm{End}(V) \;,
\end{equation}
and
\begin{equation}
\mathrm{End}^\pm(V) = \Big\{a\in\mathrm{End}(V)\;:\; \tau a = \pm a \tau \Big\} \;.
\end{equation}
Then $\mathrm{End}(V)= \mathrm{End}^+(V)\oplus\mathrm{End}^-(V)$ is a superalgebra, 
and $V$ is a super $\mathrm{End}(V)$-module.

For $a\in\mathrm{End}(V)$, its supertrace is defined by
\begin{equation}
\trs\big[a\big] = \tr\big[\tau a\big] 
\index{t rs@$\trs$} \;.
\end{equation} 
For $a,b\in\mathrm{End}(V)$, we have
\begin{equation}
\label{dilab-eq-trs-commutator-zero}
\trs\big[[a,b]\big] = 0 \;.
\end{equation} 

In this paper, 
we will apply the superalgebra formalism to the following setting.
Let $M$ be a smooth manifold. 
We denote by $\Omega^\cdot(M)$ be the algebra of differential forms on $M$. 
We always equip $\Omega^\cdot(M)$ with the $\mathbb{Z}_2$-grading $\Omega^\mathrm{even/odd}(M)$. 
Then $\Omega^\cdot(M)$ is a supercommutative superalgebra, 
i.e., $[\alpha_1,\alpha_2]=0$ for $\alpha_1,\alpha_2\in\Omega^\cdot(M)$.
Let $F$ be a complex vector bundle over $M$.
We denote by $\Omega^\cdot(M,F)$ the vector space of differential forms on $M$ with values in $F$. 
We equip $\Omega^\cdot(M,F)$ with the $\mathbb{Z}_2$-grading  $\Omega^\mathrm{even/odd}(M,F)$.
Then $\Omega^\cdot(M,F)$ is a super $\Omega^\cdot(M)$-module. 

\subsection{Clifford algebras}
\label{dilab-subsect-clifford-pre}

\

Let $V$ be a real vector space. 
Let $g^{V}$ be an Euclidean metric on $V$. 
Let 
\begin{equation}
\bigotimes V := \bigoplus_{j=0}^\infty V^{\otimes j}
\end{equation}
be the tensor algebra of $V$.

\begin{defn}
Let $I\subseteq\bigotimes V$ be the bi-ideal generated by
\begin{equation}
u \otimes v + v \otimes u + 2 g^V(u,v) \;,\hspace{5mm} u,v\in V \;.
\end{equation}
Set
\begin{equation}
C(V,g^V) = \big( \bigotimes V \big) / I \;,
\end{equation}
called the Clifford algebra associated with $(V,g^V)$. 
\end{defn}

Let 
\begin{equation}
c : V \rightarrow C(V,g^V)
\end{equation}
be the map induced by the canonical injection $V\rightarrow \bigotimes V$. 
For $u,v\in V$, we have
\begin{equation}
c(u)c(v) + c(v)c(u) = -2g^V(u,v) \;.
\end{equation}

Let $e_1,\cdots,e_n\in V$ be an orthogonal basis of $V$. 
Then
\begin{equation}
\label{dilab-eq-basis-clifford-alg}
c(e_{j_1})c(e_{j_2})\cdots c(e_{j_r}) \;,\hspace{5mm}
0 \leqslant r \leqslant n \;,\; j_1<j_2<\cdots<j_r \;,
\end{equation}
is a basis of  $C(V,g^V)$.
Let $C^\pm(V,g^V)\subseteq C(V,g^V)$ be the vector subspace spanned by the terms in 
\eqref{dilab-eq-basis-clifford-alg}
with $r$ even/odd. 
Then $C(V,g^V)$ becomes a superalgebra.

Now we suppose that $V$ is equipped with a complex structure $J\in\mathrm{End}(V)$ 
and that $g^V$ is $J$-invariant, i.e., 
$g^V(\cdot,\cdot) = g^V(J\cdot,J\cdot)$.
Set
\begin{equation}
V_\mathbb{C} = V\otimes_\mathbb{R}\mathbb{C} \;.
\end{equation}
The action of $J$ extends $\mathbb{C}$-linearly to $V_\mathbb{C}$. 
The Euclidean metric $g^V$ extends to a $\mathbb{C}$-bilinear form on $V_\mathbb{C}$.

Set
\begin{equation}
V_\mathbb{C}^{1,0} = \Big\{ v\in V_\mathbb{C}\;:\;Jv=iv\Big\} \;,\hspace{5mm}
V_\mathbb{C}^{0,1} = \Big\{ v\in V_\mathbb{C}\;:\;Jv=-iv\Big\} \;.
\end{equation}
We have 
\begin{equation}
V_\mathbb{C} = V_\mathbb{C}^{1,0} \oplus V_\mathbb{C}^{0,1} \;. 
\end{equation}
For $v\in\mathbb{V}_\mathbb{C}$, 
let $v^{(1,0)}$ (resp. $v^{(0,1)}$) be its component in $V_\mathbb{C}^{1,0}$ (resp. $V_\mathbb{C}^{0,1}$).

Let $V_\mathbb{C}^*$ be the vector space of $\mathbb{R}$-linear forms on $V_\mathbb{C}$. 
For $v\in V_\mathbb{C}$, let $v^*\in V_\mathbb{C}^*$ be its dual (with respect to $g^V$).

Set
\begin{equation}
V_\mathbb{C}^{*,1,0} = \Big\{ f\in V_\mathbb{C}^*\;:\;f\circ J=if\Big\} \;,\hspace{5mm}
V_\mathbb{C}^{*,0,1} = \Big\{ f\in V_\mathbb{C}^*\;:\;f\circ J=-if\Big\} \;.
\end{equation}
For $v\in V_\mathbb{C}^{1,0}$ (resp. $v\in V_\mathbb{C}^{0,1}$), 
we have $v^*\in V_\mathbb{C}^{*,0,1}$ (resp. $v^*\in V_\mathbb{C}^{*,1,0}$).

For $v\in V_\C$, 
we define the product operator
\begin{align}
\begin{split}
v^*\wedge : \; \Lambda^k V_\C^* & \rightarrow \Lambda^{k+1} V_\C^* \\
\alpha & \mapsto v^*\wedge\alpha\;,
\end{split}
\end{align}
and the contraction operator
\begin{align}
\begin{split}
i_v : \;\Lambda^k V_\C^* & \rightarrow \Lambda^{k-1} V_\C^* \\
\alpha & \mapsto \big((u_1,\cdots,u_{k-1}) \mapsto \alpha(v,u_1,\cdots,u_{k-1})\big) \;.
\end{split}
\end{align}
Set
\begin{align}
\label{dilab-eq-pre-clifford-complex-rep}
\begin{split}
c : V & \rightarrow \mathrm{End}\big(\Lambda^\cdot (V_\mathbb{C}^{*,0,1}) \big) \\
v & \mapsto v^{(1,0),*}\!\wedge - i_{v^{(0,1)}} \;.
\end{split}
\end{align}
For $u,v\in V$, we have
\begin{equation}
c(u)c(v)+c(v)c(u)+g^V(u,v) = 0 \;.
\end{equation}
Thus $c$ extends to a representation
\begin{equation}
c \; : \; C\big(V,\frac{1}{2}g^V\big) \rightarrow \mathrm{End}\big(\Lambda^\cdot (V_\mathbb{C}^{*,0,1}) \big) \;.
\end{equation}

\subsection{Even/odd characteristic classes}
\label{dilab-subsec-evenodd-char}
\

Let $M$ be a smooth manifold. 
Let $F$ be a complex vector bundle over $M$ of rank $r$. 

Let $\n^{F}$ be a connection on $F$. Then $\n^{F}$ induces a 
differential operator 
\begin{equation}
\n^F : \Omega^{\cdot}(M,F)\rightarrow \Omega^{\cdot+1}(M,F) \;.
\end{equation}
Let
\begin{equation}
\n^{F,2} \in \Omega^2(M,\mathrm{End}(F))
\end{equation}
be the curvature of $\n^F$.

For $\omega\in\Omega^k(M)$, put
\begin{equation}
\varphi\omega = (2\pi i)^{-k/2} \omega \;.
\end{equation}

Let $\tr\big[\cdot\big] : \mathrm{End}(F) \rightarrow \mathbb{C}$ be the trace map, 
which extends to
\begin{equation}
\tr\big[\cdot\big] : \Omega^\cdot(M,\mathrm{End}(F)) \rightarrow \Omega^\cdot(M) 
\end{equation}
such that for $\alpha\in\Omega^\cdot(M)$, $A\in\smooth(M,\mathrm{End}(F))$,
\begin{equation}
\tr\big[\omega A\big] = \omega \tr\big[ A \big].
\end{equation}

Let $P$ be an invariant polynomial on $\mathfrak{gl}(r,\C)$.

\begin{thm}[Chern-Weil]
\label{dilab-intro-thm-chern-weil}
The differential form 
\begin{equation}
\varphi P(-\nabla^{F,2})\in\Omega^\mathrm{even}(M)
\end{equation}
is closed. 
The cohomology class 
\begin{equation}
P(F) := \left[ \varphi P(-\nabla^{F,2}) \right] \in H^\mathrm{even}(M)
\end{equation}
is independent of $\n^F$.
\end{thm}

Now we assume that $\n^F$ is a flat connection, i.e., $\n^{F,2} = 0$. 
Then 
\begin{equation}
\varphi P(-\nabla^{F,2}) = P(0) 
\end{equation} 
is just a constant function on $M$.

For flat vector bundles, there are non trivial characteristic classes of odd degree. 
We will follow the construction of Bismut-Lott \cite[\textsection 1]{bl}. 

Let $g^F$ be a Hermitian metric on $F$. 
Let $\n^{F,*}$ be the adjoint connection, i.e.,
for $\sigma_1,\sigma_2\in\smooth(M,F)$ and $U\in\smooth(M,TM)$, we have
\begin{equation}
g^F(\n^F_U \sigma_1,\sigma_2) + g^F(\sigma_1,\n^{F,*}_U \sigma_2) = U g^F(\sigma_1,\sigma_2) \;.
\end{equation} 
Then $\n^{F,*,2} = 0$.

Set
\begin{equation}
\omega^F = \n^{F,*} - \n^F \in \Omega^1(M,\mathrm{End}(F)) \;.
\end{equation} 

Let $f$ be an odd polynomial in one variable with complex coefficients.
Set
\begin{equation}
f(F,\n^F,g^F) = \sqrt{2\pi i} \varphi \tr \big[f(\omega^F/2)\big] \in \Omega^\mathrm{odd}(M) \;.
\end{equation}

The following theorem was established by Bismut-Lott \cite[Theorem 1.8]{bl}.

\begin{thm}
The differential form 
\begin{equation}
f(F,\n^F,g^F) \in \Omega^\mathrm{odd}(M)
\end{equation}
is closed. 
The cohomology class
\begin{equation}
f(F,\n^F) := \left[ f(F,\n^F,g^F) \right] \in H^\mathrm{odd}(M)
\end{equation}
is independent of $g^F$.
\end{thm}

\begin{rem}
If $f$ is an even polynomial, 
by \cite[Proposition 1.3]{bl}, 
we have
\begin{equation}
\tr \big[f(\omega^F)\big] = f(0) r \;.
\end{equation}
\end{rem}

\subsection{Fibrations equipped with a connection and a fiberwise metric}
\label{dilab-subsec-fibration-connection-metric}

\

Let $\pi : \mathcal{N} \rightarrow M$ be a smooth fibration with compact 
fiber $N$.

Let $TN$ be the relative tangent bundle of the fibration.
We equip the fibration with a connection, i.e., 
a smooth splitting
\begin{equation}
\label{dilab-eq-pre-connection-fibration}
T\mathcal{N} = T^H\mathcal{N} \oplus TN \;.
\end{equation}
Then
$T^H\mathcal{N}\simeq\pi^*TM$. 
Let
\begin{equation}
P^{TN} : T\mathcal{N} \rightarrow TN \;,\hspace{5mm}
P^{T^H\mathcal{N}} : T\mathcal{N} \rightarrow T^H\mathcal{N}
\end{equation}
be the projections. 
For $U\in TM$, 
let $U^H\in T^H\mathcal{N}$ be the lift of $U$, i.e., $\pi_*U^H = U$.
 
For $U,V$ vector fields on $M$, set
\begin{equation}
\label{dilab-eq-pre-def-T}
T(U,V) = [U,V]^H - [U^H,V^H] \;.
\end{equation}
We have $T\in\Omega^2(M,\smooth(N,TN))$. 
We call $T$ the curvature of the fibration.

We equip $TM$ and $TN$ with Riemannian metrics $g^{TM}$ and $g^{TN}$.
Let $\pi^* g^{TM}$ be the induced metric on $T^H\mathcal{N}$.
Set
\begin{equation}
g^{T\mathcal{N}} = \pi^* g^{TM} \oplus g^{TN} \;,
\end{equation} 
which is a Riemannian metric on $g^{T\mathcal{N}}$.  
Let $\left\langle\cdot,\cdot\right\rangle$ be the corresponding scalar product.

Let $\n^{T\mathcal{N}}$ be the Levi-Civita connection on $T\mathcal{N}$ associated with $g^{T\mathcal{N}}$. 

\begin{defn}
Let $\nabla^{TN}$ be the connection on $TN$ defined by
\begin{equation}
\n^{TN} = P^{TN} \n^{T\mathcal{N}} P^{TN}.
\end{equation}
\end{defn}
Then  $\n^{TN}$ is independent of  $g^{TM}$ (cf. \cite[\textsection 1(c)]{b}). 

Let $L_\cdot$ be the Lie derivative.
For $U$ a vector field on $M$, set
\begin{equation}
\omega^{TN}(U) = (g^{TN})^{-1} L_{U^H} g^{TN} \in \smooth(\mathcal{N},\mathrm{End}(TN)) \;.
\end{equation} 
If $V\in TN$, then $\n^{TN}_V$ coincides with the usual Levi-Civita 
connection along the fiber $N$.
If $U\in TM$, then (cf. \cite[\textsection 1(c)]{b})
\begin{equation}\label{eq:corr1}
\n^{TN}_{U^H} = L_{U^H} + \frac{1}{2}\omega^{TN}(U) \;.
\end{equation}

Put
\begin{equation}
\n^{T\mathcal{N},\oplus} = P^{TN} \n^{T\mathcal{N}} P^{TN} \oplus P^{T^H\mathcal{N}} \n^{T\mathcal{N}} 
P^{T^H\mathcal{N}}.
\end{equation}

\begin{defn}\label{defSTX}
For $U\in T\mathcal{N}$, set
\begin{equation}
\label{dilab-eq-pre-second-form}
S^{TN}(U) = \nabla^{T\mathcal{N}}_U - \nabla^{T\mathcal{N},\oplus}_U 
\in \smooth(\mathcal{N}, \mathrm{End}(T\mathcal{N})) \;.
\end{equation}
\end{defn}
Then $\left\langle  S^{TN}\left(\cdot\right)\cdot,\cdot\right\rangle$ 
is independent of $g^{TM}$ (cf. \cite[\textsection 1(c)]{b}).

\section{The Chern-Weil theory of a flat fibration}
\label{dilab-sec-charac}

The purpose of this section is to construct characteristic 
classes and characteristic forms on the total space of a flat 
fibration with compact complex fibers.
This section is organized as follows. 

In \textsection \ref{dilab-subsec-conseq-chern-weil}, 
we state a consequence of the Chern-Weil theory, 
which will be of constant use in the rest of this section.

In \textsection \ref{dilab-subsection-a-flat-fibration}, 
we define a flat fibration with complex fibers.

In \textsection \ref{dilab-subsection-a-vector-and-superconnection}, 
we construct a complex vector bundle $E$ over the total space of fibration.

In \textsection \ref{dilab-subsection-the-connections-and-supperconnections-to-metric},
we construct connections on $E$. 
In particular, given a Hermitian metric on $E$, 
we construct a unitary connection on $E$ 
and show that the integral along the fiber of the usual 
Chern-Weil forms associated with this connection vanishes in positive 
degree.

In \textsection \ref{dilab-subsection-chern-simons-char}, 
we construct odd characteristic forms. 
These characteristic forms will appear on the right-hand side 
of the Riemann-Roch-Grothendieck formula in \textsection \ref{dilab-section-rrg}.

In \textsection \ref{dilab-subsect-mul-odd-ch}, 
we construct a natural  multiplication of the odd characteristic forms. 

\subsection{A consequence of  Chern-Weil theory}
\label{dilab-subsec-conseq-chern-weil}

\

Let $N$ be a smooth compact oriented manifold. 
Let $\big(\Omega^{\cdot}(N),d_N\big)$ 
be the de Rham complex of smooth differential forms on $N$.
We denote by $H^{\cdot}(N)$ its cohomology. 

Let $V$ be a finite dimensional real vector space. 

We will replace the de Rham complex  $\big(\Omega^{\cdot}(N),d_N\big)$ 
by the twisted de Rham complex $\big(\Omega^{\cdot}(N,\Lambda^\cdot (V^*)),d_N\big)$.
Its cohomology is equal to $H^{\cdot}(N)\widehat{\otimes}\Lambda^{\cdot}(V^*)$. 

Let $\big(\Omega^{\cdot}(N\times V), d_{N\times V}\big) $ be the de Rham complex of $N\times V$. 
Then $\big(\Omega^{\cdot}(N,\Lambda^\cdot (V^*)),d_N\big)$ can 
be identified with the subcomplex of $\big(\Omega^{\cdot}(N\times V), d_{N\times V}\big) $ 
that consists of forms which are constant along $V$. 

Let $p : N\times V \rightarrow N$ and $q : N\times V \rightarrow V$ 
be the  natural projections. 
Let $q_{*}$ denote the integral along the fiber $N$, 
i.e., for $\alpha\in\Omega^\cdot(V)$ and $\beta\in\Omega^\cdot(N)$,
\begin{equation}
q_*[\alpha\wedge\beta] = \alpha\int_N\beta \;,
\end{equation}
By restricting $q_*$ to forms which are constant along $V$, we get a 
map
\begin{equation}
q_* : \Omega^{\cdot}(N,\Lambda^\cdot (V^*)) \rightarrow \Lambda^\cdot (V^*) \;.
\end{equation}

Let $E$ be a complex vector bundle of rank $r$ over $N$. 
Let $\nabla^{E}$ be a connection on $E$. 
Its curvature $\nabla^{E,2}$ is a smooth section of $\Lambda^2(T^*N) \otimes \text{End}(E)$. 
The vector bundle $E$ lifts to the vector bundle $p^*E$ on $N\times V$, 
and $\nabla^{E} $ lifts to a connection on $p^*E$, 
which is still denoted by $\nabla^E$. 
Let $S$ be a smooth section on $N$ of $V^*\otimes \text{End}(E)$. 
We can view $S$ as a section of $V^{*}\otimes\text{End}(E)$ on $N\times V$, 
which is constant along $V$. 
Then $\nabla^{E}+S$ is also a connection on $p^*E$.  
Its curvature $(\nabla^E+S)^2$ is a smooth section 
of $\Big(\Lambda^\cdot(T^*N)\widehat{\otimes}\Lambda^{\cdot}(V^*)\Big)^\mathrm{even}\otimes\text{End}(E)$ over $N\times V$, 
which is constant along $V$.

The following proposition is a direct consequence of  Chern-Weil 
theory.

\begin{prop}
\label{dilab-prop-conseq-chern-weil}
For any invariant complex polynomial $P$ on $\mathfrak{gl}(r,\mathbb{C})$, 
\begin{equation}
P\big(-(\nabla^E+S)^2\big) \in \Omega^\cdot(N,\Lambda^\cdot (V^*)) 
\end{equation}
is closed. 
Its cohomology class
\begin{equation}
\left[P\big(-(\nabla^E+S)^2\big)\right] \in H^\cdot(N)\widehat{\otimes}\Lambda^\cdot(V^*)
\end{equation}
is independent of $\nabla^E$ and $S$. 
In particular,
\begin{equation}
\left[P\big(-(\nabla^E+S)^2\big)\right] \in H^\cdot(N)\subseteq H^\cdot(N)\widehat{\otimes}\Lambda^\cdot(V^*) \;.
\end{equation}
\end{prop}

\subsection{A flat complex fibration}
\label{dilab-subsection-a-flat-fibration}

\

Let $G \index{G@$G$}$ be a Lie group. Let $N \index{Nm@$N$}$ be a compact complex manifold of dimension $n$. 
We assume that $G$ acts holomorphically on $N$. 
Let $M \index{M@$M$}$  be a real manifold. Let $p : P_G\rightarrow M \index{PG@$P_G$}$  be a principal $G$-bundle equipped with a connection. 
Set
\begin{equation}
\mathcal{N} = P_G\times_G N \;. 
\index{Nmathcal@$\mathcal{N}$}
\end{equation} 
Let $q : \mathcal{N}\rightarrow M \index{q@$q$}$ be the natural projection, 
which defines a fibration with  canonical fiber $N$. 

Let $T_\mathbb{R}N$ be the real tangent bundle of $N$.
Set $T_\mathbb{C}N = T_\mathbb{C}N \otimes_\mathbb{R} \mathbb{C}$.

The connection on $P_G$ induces a connection on the fibration $q : \mathcal{N}\rightarrow M$, 
i.e., a splitting
\begin{equation}
\label{dilab-eq-splitting-tangent-bundle}
T\mathcal{N} = T_{\mathbb{R}}N \oplus T^H\mathcal{N} \index{TNmathcalH@$T^H\mathcal{N}$} \;.
\end{equation}
Then $T^H\mathcal{N}\simeq q^*TM$. 
The splitting \eqref{dilab-eq-splitting-tangent-bundle} induces the following identification
\begin{equation}
\label{dilab-eq-double-splitting-forms}
\Lambda^\cdot(T_\mathbb{C}^*\mathcal{N}) = 
\Lambda^\cdot(T^*_\mathbb{C}N) \widehat{\otimes}
q^*\Lambda^\cdot(T^*_\mathbb{C}M) \;.
\end{equation}
Let $TN$ be the holomorphic tangent bundle of $N$. 
Using the splitting $T_\mathbb{C}N = TN \oplus \overline{TN}$, 
we get a further splitting
\begin{equation}
\label{dilab-eq-triple-splitting-forms}
\Lambda^\cdot(T_\mathbb{C}^*\mathcal{N}) = 
\Lambda^\cdot(T^*N) \widehat{\otimes}
\Lambda^\cdot(\overline{T^*N}) \widehat{\otimes}
q^*\Lambda^\cdot(T^*_\mathbb{C}M) \;.
\end{equation}
Put
\begin{equation}
\label{dilab-eq2-triple-splitting-forms}
\Omega^{(p,q,r)}(\mathcal{N}) = \smooth\big(\mathcal{N},
\Lambda^p(T^*N) \widehat{\otimes}
\Lambda^q(\overline{T^*N}) \widehat{\otimes}
q^*\Lambda^r(T^*_\mathbb{C}M)\big) 
\index{OmegaNmathcal@$\Omega^{(p,q,r)}(\mathcal{N})$} \;.
\end{equation} 
Then
\begin{equation}
\label{dilab-eq3-triple-splitting-forms}
\Omega^k(\mathcal{N}) = \bigoplus_{p+q+r=k} \Omega^{(p,q,r)}(\mathcal{N}) \;.
\end{equation}

In the sequel, we assume that the connection on $P_G$ is flat. 
Then $q:\mathcal{N}\rightarrow M$ is a flat fibration, 
i.e., its curvature $T=0$ (cf. \eqref{dilab-eq-pre-def-T}).

Let $d_N$ be the de Rham operator on $\Omega^\cdot(N)$.
Let $d_M$ be the de Rham operator on $\Omega^\cdot(M)$,
which lifts to $\Omega^\cdot(\mathcal{N})$ in the following sense : 
let $(f_\alpha)$ be a basis of $TM$, 
let $(f^\alpha)$ be the dual basis of $T^*M$.
then
\begin{equation}
d_M = \sum_\alpha (q^*f^\alpha)\wedge L_{f_\alpha^H} \;.
\end{equation}
Let $d_{\mathcal{N}}$ \index{dNmathcal@$d_{\mathcal{N}}$} be the de Rham operator on $\mathcal{N}$. 
Since $T=0$, by \cite[Proposition 3.4]{bl}, 
we get 
\begin{equation}
\label{dilab-eq-double-splitting-derham}
d_{\mathcal{N}} = d_N + d_M 
\index{dN@$d_N$} \index{dM@$d_M$} \;.
\end{equation}

Let $\partial_N \index{dN partial@$\partial_N$}$  (resp. $\overline{\partial}_N \index{dN partialoverline@$\overline{\partial}_N$}$) 
be the holomorphic (resp. anti-holomorphic) Dolbeault operator on $N$.  
We have
\begin{equation}
\label{dilab-eq-vertical-splitting-derham}
d_N = \partial_N + \overline{\partial}_N \;.
\end{equation} 

By (\ref{dilab-eq-double-splitting-derham}) and (\ref{dilab-eq-vertical-splitting-derham}), we get
\begin{equation}
\label{dilab-eq-triple-splitting-derham}
d_{\mathcal{N}} = \partial_N + \overline{\partial}_N + d_M \;.
\end{equation}

The following  relations hold,
\begin{align}
\begin{split}
& d_M^2 = d_N^2 = \partial_N^2 = \overline{\partial}_N^2 = 0 \;,\\
& \big[d_M,d_N\big] = \big[d_M,\partial_N\big] = \big[d_M,\overline{\partial}_N\big] =
\big[d_N,\partial_N\big] = \big[d_N,\overline{\partial}_N\big] = \big[\partial_N,\overline{\partial}_N\big] = 0 \;.
\end{split}
\end{align}

\subsection{A fiberwise holomorphic vector bundle}
\label{dilab-subsection-a-vector-and-superconnection}

\

Let $E_0\index{E0@$E_0$}$ be a holomorphic vector bundle over $N$ of 
rank $r$. We assume that the action of $G$ on $N$ lifts to a 
holomorphic action on $E_{0}$.
Set
\begin{equation}
E = P_G \times_G E_0 
\index{E@$E$}\;,
\end{equation}
which is a complex vector bundle over $\mathcal{N}$. 
Furthermore, $E$ is holomorphic along $N$.

Let $\overline{\partial}_N^E \index{dN partialoverline E@$\overline{\partial}_N^E$}$ be the fiberwise holomorphic structure of $E$. 
Let $d_M^E \index{dM E@$d_M^E$}$ be the lift of the de Rham operator on $M$ to 
$\Omega^\cdot(\mathcal{N},E)$.
We have
\begin{equation}
\label{dilab-eq-commute-n-E-NM}
\overline{\partial}_N^{E,2} = d_M^{E,2} = \big[\overline{\partial}_N^E,d_M^E\big] = 0 \;.
\end{equation}

\subsection{Connections}
\label{dilab-subsection-the-connections-and-supperconnections-to-metric}

\

Set
\begin{equation}
\label{dilab-eq-def-flat-superconnection}
{A^E}'' = \overline{\partial}^E_N + \diff^{E}_M \;,
\index{AE prim2@${A^E}''$}
\end{equation} 
which acts on $\Omega^\cdot(\mathcal{N},E)$. 
By \eqref{dilab-eq-commute-n-E-NM}, we have
\begin{equation}
\label{dilab-eq-AEsquare=0}
\big({A^E}''\big)^2 = 0 \;.
\end{equation}

Let $\overline{E}^* $ be the anti-dual vector bundle of $E$. 
When replacing the complex structure of $N$ by the conjugate complex structure, 
$\overline{E}^*$ enjoys exactly the same properties as $E$.  
We construct 
$\partial^{\overline{E}^*}_N$
$d_M^{\overline{E}^*}$ and
${A^{\overline{E}^*}}'$ in the same way as $\overline{\partial}^E_N$, $d_M^E$ and ${A^E}''$.
In particular, 
\begin{equation}
\label{dilab-eq-def-flat-superconnection-antidual}
{A^{\overline{E}^*}}' = \partial^{\overline{E}^*}_N + d_M^{\overline{E}^*} \;.
\end{equation}
Proceeding in the same way as in 
\eqref{dilab-eq-commute-n-E-NM} and \eqref{dilab-eq-AEsquare=0}, 
we  have
\begin{equation}
\label{dilab-eq-commute-n-Edual-NM}
\partial_N^{\overline{E}^*,2} = d_M^{\overline{E}^*,2} = \big[\overline{\partial}_N^{\overline{E}^*},d_M^{\overline{E}^*}\big] = 0 \;,
\end{equation}
and
\begin{equation}
\label{dilab-eq-AEdualsquare=0}
\big({A^{\overline{E}^*}}'\big)^2 = 0 \;.
\end{equation}

Let $g^E$ be a Hermitian metric on $E$. Then $g^E$ defines an isomorphism $g^E : E \rightarrow \overline{E}^*$.
Set
\begin{equation}
\label{dilab-eq-def-EN-EMu}
\partial^E_N =  (g^E)^{-1} \partial^{\overline{E}^*}_N g^E 
\index{dN partial E @$\partial^E_N$}\;,\hspace{5mm}
d_M^{E,*} = (g^E)^{-1} d_M^{\overline{E}^*} g^E 
\index{dM E star@$d_M^{E,*}$}\;,
\end{equation}
which act on $\Omega^\cdot(\mathcal{N},E)$.
By \eqref{dilab-eq-commute-n-Edual-NM} and \eqref{dilab-eq-def-EN-EMu}, we have
\begin{equation}
\label{dilab-eq-commute-n-E-NM-star}
\partial_N^{E,2} = d_M^{E,*,2} = \big[ \overline{\partial}^E_N,d_M^{E,*} \big] = 0 \;.
\end{equation}

Set
\begin{equation}
\label{dilab-eq-def-flat-superconnection-star}
{A^E}' = (g^E)^{-1}{A^{\overline{E}^*}}'g^E= \partial^E_N + d_M^{E,*} 
\index{AE prim@${A^E}'$}\;.
\end{equation} 
Then, by \eqref{dilab-eq-commute-n-E-NM-star}, we have
\begin{equation}
\label{dilab-eq-AEprimsquare=0}
\big({A^E}'\big)^2 = 0 \;.
\end{equation}

Let $N^{\Lambda^\cdot(T^*M)} \index{Nnumber M@$N^{\Lambda^\cdot(T^*M)}$}$ be the number operator of $\Lambda^\cdot(T^*M)$.

\begin{defn}
Set
\begin{align}
\label{dilab-eq-def-superconnection}
\begin{split}
& A^E = 2^{-N^{\Lambda^\cdot(T^*M)}} \big({A^E}' + {A^E}''\big) 2^{N^{\Lambda^\cdot(T^*M)}} 
\index{AE@$A^E$} \;,\\
& B^E = 2^{-N^{\Lambda^\cdot(T^*M)}} \big({A^E}' - {A^E}''\big) 2^{N^{\Lambda^\cdot(T^*M)}} 
\index{BE@$B^E$} \;.
\end{split}
\end{align}
\end{defn}

By \eqref{dilab-eq-AEsquare=0} and \eqref{dilab-eq-AEprimsquare=0}, we have
\begin{equation}
\label{dilab-eq-relation-A-B}
A^{E,2} = 2^{-N^{\Lambda^\cdot(T^*M)}}  \big[{A^E}',{A^E}''\big] 2^{N^{\Lambda^\cdot(T^*M)}}  = - B^{E,2} \;.
\end{equation}

Set
\begin{equation}\label{eq:corr2}
d_N^E = \partial^E_N + \overline{\partial}^E_N 
\index{dN E@$d_N^E$}\;,\hspace{5mm}
d_M^{E,\mathrm{u}} = \frac{1}{2}\big(d_M^E + d_M^{E,*}\big) 
\index{dM E uni@$d_M^{E,\mathrm{u}}$}\;.
\end{equation}
Then
\begin{equation}
\label{dilab-eq-explicite-A}
A^E = d_N^E + d_M^{E,\mathrm{u}} \;.
\end{equation}
Thus $A^E$ is a Hermitian connection on $E$ over $\mathcal{N}$. 

Set
\begin{equation}
\label{dilab-eq-def-omegaE}
\omega^E = d_M^{E,*} - d_M^E = \big(g^E\big)^{-1} d_M^E g^E \in 
\smooth\big(\mathcal{N}, T^*M \otimes _{\mathbb R}\mathrm{End}(E)\big) 
\index{omegaE@$\omega^E$}\;. 
\end{equation}
Then
\begin{equation}
\label{dilab-eq-explicite-B}
B^E = \partial^E_N - \overline{\partial}^E_N + \frac{1}{2}\omega^E \;.
\end{equation}
Thus $B^E\in\Omega^\cdot\big(M,\text{End}(\Omega^\cdot(N,E))\big)$.

\begin{prop}
\label{dilab-prop-annulation-curvature}
For any invariant polynomial $P$ on $\mathfrak{gl}(r,\mathbb{C})$, 
we have
\begin{equation}
\label{dilab-eq1-prop-annulation-curvature}
\big(\partial_N-\overline{\partial}_N\big)P\big(-A^{E,2}\big) = 0 \;.
\end{equation}  
Also 
\begin{equation}
\label{dilab-eq2-prop-annulation-curvature}
P\big(-A^{E,2}\big) - P\big(-d_N^{E,2}\big) \in \im \big(\partial_N-\overline{\partial}_N\big) \;.
\end{equation} 
As a consequence, we have
\begin{equation}
\label{dilab-eq3-prop-annulation-curvature}
q_*\big[P\big(-A^{E,2}\big)\big] = q_*\big[P\big(-d_N^{E,2}\big)\big] \;,
\end{equation}
which is a constant function on $M$.
\end{prop}
\begin{proof}
Let $N^{\Lambda^\cdot(\overline{T^*N})}$ be the number operator of $\Lambda^\cdot(\overline{T^*N})$.
Set $U=(-1)^{N^{\Lambda^\cdot(\overline{T^*N})}}$. 

To establish 
\eqref{dilab-eq1-prop-annulation-curvature} and 
\eqref{dilab-eq2-prop-annulation-curvature}, 
we only need 
to show that
\begin{equation}
\label{dilab-eq01-prop-annulation-curvature}
d_N UP\big(-A^{E,2}\big) = 0 \;,
\end{equation} 
and
\begin{equation}
\label{dilab-eq02-prop-annulation-curvature}
UP\big(-A^{E,2}\big) - UP\big(-d_N^{E,2}\big) \in \im\big(d_N\big) \;.
\end{equation}

By \eqref{dilab-eq-explicite-B}, we have
\begin{equation}
U^{-1}B^EU  = d_N^E + \frac{1}{2}\omega^E \;.
\end{equation}
Now, applying \eqref{dilab-eq-relation-A-B}, we get
\begin{equation}
\label{dilab-eq-prop-annulation-curvature-1}
U^{-1}A^{E,2}U = -U^{-1}B^{E,2}U = -\big(d_N^E+\frac{1}{2}\omega^E\big)^2 \;.
\end{equation}
We may and we will assume that $P$ is  homogeneous. 
By \eqref{dilab-eq-prop-annulation-curvature-1}, we have
\begin{equation}
\label{dilab-eq-prop-annulation-curvature-2}
U P\big(-A^{E,2}\big) = (-1)^{\deg P} P\Big(-\big(d_N^E+\frac{1}{2}\omega^E\big)^2\Big) \;.
\end{equation}
Applying Proposition \ref{dilab-prop-conseq-chern-weil} to 
the right-hand side of \eqref{dilab-eq-prop-annulation-curvature-2}, 
we get \eqref{dilab-eq01-prop-annulation-curvature}. 

We decompose \eqref{dilab-eq-prop-annulation-curvature-2}
according to \eqref{dilab-eq3-triple-splitting-forms}. 
By extracting the components which are of positive degree along $M$, 
we get
\begin{align}
\label{dilab-eq-prop-annulation-curvature-3}
\begin{split}
& U P\big(-A^{E,2}\big) - U P\big(-d^{E,2}_N\big) \\
& = (-1)^{\deg P} P\Big(-\big(d_N^E+\frac{1}{2}\omega^E\big)^2\Big) 
- (-1)^{\deg P} P\Big(-d_N^{E,2}\Big) \;.
\end{split}
\end{align}
Applying Proposition \ref{dilab-prop-conseq-chern-weil} to 
the right-hand side of \eqref{dilab-eq-prop-annulation-curvature-3}, 
we get \eqref{dilab-eq02-prop-annulation-curvature}. 

Taking the integral of \eqref{dilab-eq2-prop-annulation-curvature} along $N$,
we get \eqref{dilab-eq3-prop-annulation-curvature}.
\end{proof}

For $t\in\mathbb{R}$, set
\begin{equation}
A^E_t= d_N^E + td_M^E + (1-t)d_M^{E,*} 
\index{AEt@$A^E_t$}\;.
\end{equation}
In particular,
\begin{equation}
A^E_{1/2}=A^E \;.
\end{equation}

Set
\begin{equation}
\label{dilab-eq-def-vt} 
V_t = (2-2t)^{N^{\Lambda^\cdot(T^*N)}}(2t)^{N^{\Lambda^\cdot(\overline{T^*N})}} 
\index{Vt@$V_t$}\;.
\end{equation}

\begin{lemme}
For $t \neq 0,1$, we have
\begin{equation}
\label{dilab-eq-connection-t-rescaling}
A^{E,2}_t = 4t(1-t)V_t^{-1}A^{E,2}V_t \;.
\end{equation}
\end{lemme}
\begin{proof}
By \eqref{dilab-eq-def-flat-superconnection} and \eqref{dilab-eq-def-flat-superconnection-star}, 
we have
\begin{align}
\label{dilab-prop-rescal-A-t-proof-eq-1}
\begin{split}
2tV_t^{-1}2^{-N^{\Lambda^\cdot(T^*M)}}{A^E}''2^{N^{\Lambda^\cdot(T^*M)}}V_t & = \overline{\partial}^E_N + td_M^E \;,\\
(2-2t)V_t^{-1}2^{-N^{\Lambda^\cdot(T^*M)}}{A^E}'2^{N^{\Lambda^\cdot(T^*M)}}V_t & = \partial^E_N + (1-t)d_M^{E,*} \;.
\end{split}
\end{align}
By \eqref{dilab-eq-commute-n-E-NM}, \eqref{dilab-eq-commute-n-E-NM-star}, \eqref{dilab-eq-relation-A-B}
and \eqref{dilab-prop-rescal-A-t-proof-eq-1}, we have
\begin{align}
\begin{split}
& 4t(1-t)V_t^{-1}A^{E,2}V_t \\
= \; & \Big[(2-2t)V_t^{-1}2^{-N^{\Lambda^\cdot(T^*M)}}{A^E}'2^{N^{\Lambda^\cdot(T^*M)}}V_t\,,\,2tV_t^{-1}2^{-N^{\Lambda^\cdot(T^*M)}}{A^E}''2^{N^{\Lambda^\cdot(T^*M)}}V_t\Big] \\
= \; & \Big[\partial^E_N + (1-t)d_M^{E,*}\,,\,\overline{\partial}_N^E + td_M^E\Big] \\
= \; & \Big( \partial^E_N + (1-t)d_M^{E,*} + \overline{\partial}_N^E + td_M^E \Big)^2 
= A^{E,2}_t \;.
\end{split}
\end{align}
\end{proof}

Now we extend Proposition \ref{dilab-prop-annulation-curvature} 
by considering the  extra parameter $t$.
\begin{thm}
\label{dilab-thm-gen-flat2trivialclass}
For any invariant polynomial $P$ on $\mathfrak{gl}(r,\mathbb{C})$ and $t\in\mathbb{R}$, 
we have
\begin{equation}
\label{dilab-eq-thm-gen-flat2trivialclass}
q_*\big[P\big(-A^{E,2}_t\big)\big] = 
q_*\big[P\big(-d_N^{E,2}\big)\big] \;,
\end{equation}
which is a constant function on $M$.
\end{thm}
\begin{proof}
Since $q_*\big[P\big(-A^{E,2}_t\big)\big]$ is polynomial on $t$, 
it is sufficient to consider the case $t\neq 0,1$. 

We may suppose that $P$ is homogeneous. 
By \eqref{dilab-eq-connection-t-rescaling}, we have
\begin{equation}
\label{dilab-eq1-pf-thm-gen-flat2trivialclass}
q_*\big[P\big(-A^{E,2}_t\big)\big] = \big(4t(1-t)\big)^{\deg P} q_*\big[V_t^{-1}P\big(-A^{E,2}\big)\big] \;.
\end{equation}
Applying Proposition \ref{dilab-prop-annulation-curvature} to the right-hand side of \eqref{dilab-eq1-pf-thm-gen-flat2trivialclass}, 
we get
\begin{equation}
\label{dilab-eq2-pf-thm-gen-flat2trivialclass}
q_*\big[P\big(-A^{E,2}_t\big)\big] = \big(4t(1-t)\big)^{\deg P} q_*\big[V_t^{-1}P\big(-d_N^{E,2}\big)\big] \;.
\end{equation}
Since $P\big(-d_N^{E,2}\big)$ is a $(\deg P, \deg P)$-form on $N$, 
we have
\begin{equation}
\label{dilab-eq3-pf-thm-gen-flat2trivialclass}
V_t^{-1}P\big(-d_N^{E,2}\big) = \big(4t(1-t)\big)^{-\deg P} P\big(-d_N^{E,2}\big) \;.
\end{equation}
By \eqref{dilab-eq2-pf-thm-gen-flat2trivialclass} and \eqref{dilab-eq3-pf-thm-gen-flat2trivialclass}, 
we get \eqref{dilab-eq-thm-gen-flat2trivialclass}. 
\end{proof}

\subsection{The odd characteristic forms}
\label{dilab-subsection-chern-simons-char}

\

Set
$\varphi = (2\pi i)^{-\frac{1}{2}N^{\Lambda^\cdot (T^*\mathcal{N})}} \index{p hi var@$\varphi$}$.

Let $P$ be an invariant polynomial on $\mathfrak{gl}(r,\mathbb{C})$.

\begin{defn}
For $t\in\mathbb{R}$, set
\begin{equation}
\widetilde{P}_t\big(E, g^E\big) = \sqrt{2\pi i} \varphi \; \left\langle P'\big(-A^{E,2}_t\big),\frac{\omega^E}{2} \right\rangle \in\Omega^\mathrm{odd}(\mathcal{N})
\index{Pinvt tilde@$\widetilde{P}_t(\cdot,\cdot)$}\;.
\end{equation}
\end{defn}

Here the notation $\big\langle P'(\cdot),\cdot\big\rangle$ 
was defined in \eqref{dilab-eq-intro-def-Pder}.

\begin{prop}
\label{dilab-prop-class-odd}
For $t\in\mathbb{R}$, 
the differential form
\begin{equation}
q_*\big[\widetilde{P}_t\big(E, g^E\big)\big]\in\Omega^\mathrm{odd}(M)
\end{equation} 
is closed .
Its cohomology class 
\begin{equation}
\left[ q_*\big[\widetilde{P}_t\big(E, g^E\big)\big] \right] \in H^\cdot(M)
\end{equation}
is independent of $g^E$.
\end{prop}
\begin{proof} 
We have (cf. \cite[\textsection 1.4]{bgv})
\begin{align}
\begin{split}
\sqrt{2\pi i} \varphi \frac{\partial}{\partial t} P\big(-A^{E,2}_t\big) 
= \; & -\sqrt{2\pi i} \varphi \left\langle P'\big(-A^{E,2}_t\big),\big[A^E_t,\frac{\partial}{\partial t}A^E_t\big]\right\rangle  \\
= \; & -\sqrt{2\pi i} \varphi \, d_\mathcal{N} \left\langle P'\big(-A^{E,2}_t\big),\frac{\partial}{\partial t}A^E_t\right\rangle \\ 
= \; & -d_\mathcal{N} \varphi \left\langle P'\big(-A^{E,2}_t\big),\frac{\partial}{\partial t}A^E_t\right\rangle \;.
\end{split}
\end{align}
Since
\begin{equation}
\frac{\partial}{\partial t}A^E_t = d_M^E - d_M^{E,*} = - \omega^E \;,
\end{equation}
we have
\begin{equation}
\label{dilab-eq-closeness-chern-simons}
\sqrt{2\pi i} \varphi \;\frac{\partial}{\partial t} P\big(-A^{E,2}_t\big) = 
2 d_\mathcal{N}\widetilde{P}_t\big(E, g^E\big) \;.
\end{equation}

By Proposition \ref{dilab-thm-gen-flat2trivialclass}, we get
\begin{equation}
\label{dilab-proof-prop-class-odd-eq-1}
\frac{\partial}{\partial t} q_*\big[P\big(-A^{E,2}_t\big)\big] = 0 \;.
\end{equation}
By (\ref{dilab-eq-closeness-chern-simons}) and (\ref{dilab-proof-prop-class-odd-eq-1}), we get
\begin{equation}
d_M q_*\big[\widetilde{P}_t\big(E, g^E\big)\big] = 
q_*\big[d_\mathcal{N}\widetilde{P}_t\big(E, g^E\big)\big] = 0 \;.
\end{equation}
Thus $q_*\big[\widetilde{P}_t\big(E, g^E\big)\big]$ is closed.

The fact that 
$\left[q_*\big[\widetilde{P}_t\big(E, g^E\big)\big]\right]\in 
H^\cdot(M)$ is independent of $g^E$ 
comes from the functoriality of our construction (cf. \cite[\textsection 1.5]{bgv}). 
\end{proof}

Now we study the dependence of $\widetilde{P}_t\big(E, g^E\big)$ on $t$.

Recall that $V_t$ was defined in (\ref{dilab-eq-def-vt}).

\begin{prop}
If $P$ is homogeneous, for $t\in\mathbb{R}$, we have
\begin{equation}
\label{dilab-eq-chern-simons-form-t-rescaling}
\widetilde{P}_t\big(E, g^E\big) = \big(4t(1-t)\big)^{\deg P - 1}V_t^{-1}
\widetilde{P}_\frac{1}{2}\big(E, g^E\big) \;.
\end{equation}
In particular,
\begin{equation}
\label{dilab-eq-chern-simons-class-t-rescaling}
q_*\big[\widetilde{P}_t\big(E, g^E\big)\big] = \big(4t(1-t)\big)^{\deg P - n - 1}
q_*\big[\widetilde{P}_\frac{1}{2}\big(E, g^E\big)\big] \;.
\end{equation}
\end{prop}
\begin{proof}
Since \eqref{dilab-eq-chern-simons-form-t-rescaling} is a rational 
function of $t$, 
it is sufficient to consider the case $t\neq 0,1$.

By \eqref{dilab-eq-connection-t-rescaling}, we have
\begin{align}
\begin{split}
\left\langle P'\big(-A^{E,2}_t\big)\,,\,\frac{\omega^E}{2}\right\rangle
& = \left\langle P'\big(-4t(1-t)V_t^{-1}A^{E,2}_{\frac{1}{2}}V_t\big)\,,\,\frac{\omega^E}{2}\right\rangle \\
& = \big(4t(1-t)\big)^{\deg P'}V_t^{-1}\left\langle P'\big(-A^{E,2}_{\frac{1}{2}}\big)\,,\,\frac{\omega^E}{2}\right\rangle \\
& = \big(4t(1-t)\big)^{\deg P - 1}V_t^{-1}\left\langle P'\big(-A^{E,2}_{\frac{1}{2}}\big)\,,\,\frac{\omega^E}{2}\right\rangle \;,
\end{split}
\end{align}
which is equivalent to \eqref{dilab-eq-chern-simons-form-t-rescaling}.
\end{proof}

In the sequel, we use the convention
\begin{equation}
\widetilde{P}\big(E, g^E\big) = \widetilde{P}_\frac{1}{2}\big(E, g^E\big) 
\index{Pinv tilde@$\widetilde{P}(\cdot,\cdot)$} \;.
\end{equation}

The following proposition is a refinement of Proposition 
\ref{dilab-prop-class-odd}.

\begin{prop}
We have
\begin{align}
\label{dilab-eq-prop-class-odd-close-explicite-1}
\begin{split}
     & d_\mathcal{N} \widetilde{P}\big(E, g^E\big) \\
= \; & \frac{\sqrt{2\pi i}}{2}\varphi\; \big(\frac{\partial}{\partial t}V_t^{-1}\big)_{t=\frac{1}{2}} \big(\partial_N - \overline{\partial}_N\big) 
\int_0^1 \left\langle P'\Big(\big(\partial^E_N-\overline{\partial}^E_N+\frac{s\omega^E}{2}\big)^2\Big),\frac{\omega^E}{2}\right\rangle ds \;.
\end{split}
\end{align}
In particular, for $p = 0,\cdots,n$, we have
\begin{equation}
\label{dilab-eq-prop-class-odd-close-explicite-2}
\Big\{d_\mathcal{N}\widetilde{P}\big(E, g^E\big)\Big\}^{(p,p,\cdot)} = 0 \;.
\end{equation}
\end{prop}
\begin{proof}
By \eqref{dilab-eq-connection-t-rescaling}, we have
\begin{align}
\label{dilab-eq-pf-prop-class-odd-close-explicite-0}
\begin{split}
& \frac{\partial}{\partial t} \Big\{ \sqrt{2\pi i} \varphi\; P\big(-A^{E,2}_t\big) \Big\}_{t=\frac{1}{2}} \\
= \; & \frac{\partial}{\partial t} \Big\{ \sqrt{2\pi i} \varphi\; \big(4t(1-t)\big)^{\deg P}V_t^{-1}P\big(-A^{E,2}\big) \Big\}_{t=\frac{1}{2}} \;.
\end{split}
\end{align}
By \eqref{dilab-eq3-pf-thm-gen-flat2trivialclass} and \eqref{dilab-eq-pf-prop-class-odd-close-explicite-0}, 
we have 
\begin{align}
\label{dilab-eq-pf-prop-class-odd-close-explicite-1}
\begin{split}
& \frac{\partial}{\partial t} \Big\{ \sqrt{2\pi i} \varphi\; P\big(-A^{E,2}_t\big) \Big\}_{t=\frac{1}{2}} \\
= \; & \frac{\partial}{\partial t} \Big\{ \sqrt{2\pi i} \varphi\; \big(4t(1-t)\big)^{\deg P}V_t^{-1} 
\Big(P\big(-A^{E,2}\big) - P\big(-d_N^{E,2}\big)\Big) \Big\}_{t=\frac{1}{2}}
\end{split}
\end{align}

By \eqref{dilab-eq-prop-annulation-curvature-1}, we have
\begin{equation}
\label{dilab-eq-pf-prop-class-odd-close-explicite-2}
P\big(-A^{E,2}\big) - P\big(-d_N^{E,2}\big)
= U \left( P\Big(\big(d_N^E+\frac{\omega^E}{2}\big)^2\Big) - P\big(d_N^{E,2}\big) \right) \;.
\end{equation}
As a consequence of Proposition \ref{dilab-prop-conseq-chern-weil} (cf. \cite[\textsection 1.5]{bgv}). , we get
\begin{equation}
\label{dilab-eq-pf-prop-class-odd-close-explicite-3}
P\Big(\big(d_N^E+\frac{\omega^E}{2}\big)^2\Big) - P\big(d_N^{E,2}\big) 
= d_N \int_0^1 \left\langle P'\Big(\big(d_N^E+\frac{s\omega^E}{2}\big)^2\Big),\frac{\omega^E}{2} \right\rangle ds \;.
\end{equation}
Then
\begin{align}
\label{dilab-eq-pf-prop-class-odd-close-explicite-4}
\begin{split}
     & U \left( P\Big(\big(d_N^E+\frac{\omega^E}{2}\big)^2\Big) - P\big(d_N^{E,2}\big) \right) \\
= \; &  \big(\partial_N - \overline{\partial}_N\big) \int_0^1 
\left\langle P'\Big(\big(\partial^E_N-\overline{\partial}^E_N+\frac{s\omega^E}{2}\big)^2\Big),\frac{\omega^E}{2} \right\rangle ds \;.
\end{split}
\end{align}

By 
\eqref{dilab-eq-closeness-chern-simons}, 
\eqref{dilab-eq-pf-prop-class-odd-close-explicite-1}, 
\eqref{dilab-eq-pf-prop-class-odd-close-explicite-2} and 
\eqref{dilab-eq-pf-prop-class-odd-close-explicite-4}, 
we get (\ref{dilab-eq-prop-class-odd-close-explicite-1}).

For $p=0,\cdots,n$, we have
\begin{equation}
V_t^{-1}|_{\Omega^{(p,p,\cdot)}}=(4t(1-t))^{-p} \;,
\end{equation} whose derivative at $t=\frac{1}{2}$ is zero. This proves (\ref{dilab-eq-prop-class-odd-close-explicite-2}).
\end{proof}

\subsection{Multiplication of odd characteristic forms}
\label{dilab-subsect-mul-odd-ch}

\ 

Put
\begin{equation}
P\big(E, g^E\big) = \varphi P\big(-A^{E,2}_\frac{1}{2}\big) \;. 
\end{equation}

\begin{prop}
\label{dilab-prop-multi-cs}
Let $P,Q$ be two invariant polynomials.
The following identity holds
\begin{equation}
\label{dilab-eq-prop-multi-cs}
\widetilde{PQ}\big(E, g^E\big) =  \widetilde{P}\big(E, g^E\big) \wedge Q\big(E, g^E\big) 
                                     +  P\big(E, g^E\big) \wedge \widetilde{Q}\big(E, g^E\big) \;.
\end{equation}
\end{prop}
\begin{proof}
We have
\begin{align}
\label{dilab-proof-prop-multi-cs-eq-1}
\begin{split}
     & \left\langle (PQ)'\big(-A^{E,2}\big),\frac{\omega^E}{2} \right\rangle \\
= \; & \left\langle P'\big(-A^{E,2}\big),\frac{\omega^E}{2} \right\rangle \wedge Q\big(-A^{E,2}\big) + 
                    P\big(-A^{E,2}\big) \wedge \left\langle Q'\big(-A^{E,2}\big),\frac{\omega^E}{2} \right\rangle \;,
\end{split}
\end{align}
which implies \eqref{dilab-eq-prop-multi-cs}.
\end{proof}

For $(\alpha,\tilde{\alpha}), (\beta,\tilde{\beta})\in\Omega^{\mathrm{even}}(\mathcal{N})
\times\Omega^{\mathrm{odd}}(\mathcal{N})$, put
\begin{equation}
(\alpha,\tilde{\alpha})\cdot(\beta,\tilde{\beta}) = (\alpha\wedge\beta,\tilde{\alpha}\wedge\beta+\alpha\wedge\tilde{\beta}) \;.
\end{equation}
Then
$\big(\Omega^{\mathrm{even}}(\mathcal{N})\times\Omega^{\mathrm{odd}}(\mathcal{N}),\,+\,,\,\cdot\;\big)$
is a commutative ring. 

Let $\left( \mathbb{C}\big[\mathfrak{gl}(r,\mathbb{C})\big] \right)^{\mathrm{GL}(r,\mathbb{C})}$
be the ring of invariant polynomials on $\mathfrak{gl}(r,\mathbb{C})$.

\begin{prop}
\label{dilab-prop-mul-odd-char}
The following map is a ring homomorphism.
\begin{align}
\begin{split}
\left( \mathbb{C}\big[\mathfrak{gl}(r,\mathbb{C})\big] \right)^{\mathrm{GL}(r,\mathbb{C})} & \rightarrow 
\Omega^{\mathrm{even}}(\mathcal{N})\times\Omega^{\mathrm {odd}}(\mathcal{N}) \\
            P & \mapsto \left(P\big(E, g^E\big),\widetilde{P}\big(E, g^E\big)\right) \;.
\end{split}
\end{align}
\end{prop}
\begin{proof}
This is a direct consequence of Proposition \ref{dilab-prop-multi-cs}.
\end{proof}

Let $F$ be another complex vector bundle over $\mathcal{N}$ satisfying the same properties as $E$. 
Let $r'$ be the rank of $F$. 
Let $g^F$ be a Hermitian metric on $F$.
Let $Q$ be an invariant polynomial on $\mathfrak{gl}(r',\mathbb{C})$.

\begin{defn}
We define
\begin{equation}
\widetilde{P}\big(E, g^E\big)*\widetilde{Q}\big(F, g^F\big) = 
\widetilde{P}\big(E, g^E\big)Q\big(F, g^F\big) +
P\big(E, g^E\big)\widetilde{Q}\big(F, g^F\big) \;.
\end{equation}
\end{defn}

\begin{prop}
\label{dilab-prop-mul-odd-char-EF}
The differential form
\begin{equation}
q_*\big[\widetilde{P}\big(E, g^E\big)*\widetilde{Q}\big(F, g^F\big)\big] \in \Omega^\mathrm{odd}(M) 
\end{equation}
is closed. Its cohomology class is independent of $g^E$ and $g^F$.
\end{prop}
\begin{proof}
The argument leading to Proposition \ref{dilab-prop-class-odd} still works: 
the key step is to show that
\begin{equation}
2 d_\mathcal{N} \widetilde{P}\big(E, g^E\big)*\widetilde{Q}\big(F, g^F\big) =
\sqrt{2\pi i}\varphi \frac{\partial}{\partial t} \Big( P(-A^{E,2}_t)Q(-A^{F,2}_t) \Big)_{t=1/2} \;.
\end{equation}
\end{proof}

\section{A Riemann-Roch-Grothendieck formula}
\label{dilab-section-rrg}

The purpose of this section is 
to establish a Riemann-Roch-Grothendieck formula, 
that express the odd Chern classes associated with the 
flat vector bundle $H^{\cdot}\left(N,E\right)$ in  terms of the 
exotic characteristic classes that were defined in \textsection \ref{dilab-subsection-chern-simons-char}.
This section is organized as follows.

In \textsection \ref{dilab-subsection-flat-superconnection}, 
we introduce the infinite dimensional flat vector bundle  
$\mathscr{E}=\Omega^{(0,\cdot)}(N,E)$.

In \textsection \ref{dilab-subsec-vertical-metrics}, 
we equip $TN$ with a fiberwise K{\"a}hler metric, $E$ with a Hermitian metric. 

In \textsection \ref{dilab-subsec-levi-civita-supperconnection}, 
we introduce the Levi-Civita superconnection on $\mathscr{E}$.

In \textsection \ref{dilab-subsec-index-bundle}, 
we define the index bundle, which is the fiberwise Dolbeault cohomology of $E$. 
We show that the even characteristic form of the index bundle 
is a constant function on $M$.

In \textsection \ref{dilab-subsec-riemann-roch-grothendieck}, 
we construct differential forms $\alpha_t$, $\beta_t$ 
in the same way as \cite[\textsection 3(h)]{bl}. 
We state explicit formulas calculating their asymptotics 
as $t\rightarrow\infty$ and $t\rightarrow 0$.
As a consequence of these formulas, 
we obtain a Riemann-Roch-Grothendieck formula.

In \textsection \ref{dilab-subsec-intermediate-rrg}, 
we prove the asymptotics of $\alpha_t$, $\beta_t$ stated in \textsection \ref{dilab-subsec-riemann-roch-grothendieck}. 
The techniques applied in the proof were  initiated 
by Bismut-Gillet-Soul{\'e} \cite[\textsection 1(h)]{bgs3} and 
Bismut-K\"{o}hler \cite{bk}. 
The key idea is a Lichnerowicz formula involving additional 
Grassmannian variables $da$, $d\bar{a}$. 

In \textsection \ref{dilab-subsect-torsionform}, 
following \cite[\textsection 3(j)]{bl}, we construct  analytic 
torsion forms 
on  $M$, that transgress the R.R.G. formula at the level of 
differential forms.

\subsection{A flat superconnection and its dual}
\label{dilab-subsection-flat-superconnection}

\

Set
\begin{equation}
\label{dilab-eq-def-inf-E}
\mathscr{E}^q = \mathscr{C}^\infty(N,\Lambda^q(\overline{T^*N})\otimes E) \;,\hspace{5mm} 
\mathscr{E} = \bigoplus_q \mathscr{E}^q \index{Emathscr@$\mathscr{E}$}\;.
\end{equation}
Then $\mathscr{E}$ is an infinite dimensional flat vector bundle over $M$. 
By \eqref{dilab-eq2-triple-splitting-forms}, 
we have the identification 
\begin{equation}
\Omega^\cdot(M,\mathscr{E}) = \Omega^{(0,\cdot,\cdot)}(\mathcal{N},E) \;.
\end{equation} 

Let $\n^\mathscr{E} \index{nablaEmathscr@$\n^\mathscr{E}$}$ be the restriction of $d_M^E$ to $\Omega^\cdot(M,\mathscr{E})$. 
Then $\n^\mathscr{E}$ is the canonical flat connection on $\mathscr{E}$.

Set
\begin{equation}
\label{dilab-eq-decomposition-superconnection-infi-dim}
A^\mathscr{E} = \overline{\partial}^E_N + \n^\mathscr{E} \index{AEmathscr@$A^\mathscr{E}$} \;,
\end{equation}
which acts on $\Omega^\cdot(M,\mathscr{E})$. 
Then $A^\mathscr{E}$ is a superconnection on $\mathscr{E}$.

Recall that the operator ${A^E}''$ on $\Omega^\cdot(\mathcal{N},E)$ was defined in \eqref{dilab-eq-def-flat-superconnection}.
We have
\begin{equation}
A^\mathscr{E} = 
{A^E}''\big|_{\Omega^{(0,\cdot,\cdot)}(\mathcal{N},E)} \;.
\end{equation}
Then, by \eqref{dilab-eq-AEsquare=0}, we have
\begin{equation}
\label{dilab-eq-Asquare=0}
A^{\mathscr{E},2} = 0 \;,
\end{equation}
which is equivalent to the following identities
\begin{equation}
\overline{\partial}^{E,2}_N = \n^{\mathscr{E},2} = \big[ \overline{\partial}^E_N , \n^\mathscr{E} \big] = 0 \;.
\end{equation}

Set
\begin{equation}
\overline{\mathscr{E}}^* = \mathscr{C}^\infty(N,\Lambda^\cdot(T^*N)\otimes\Lambda^n(\overline{T^*N})\otimes \overline{E}^*) \index{Emathscr dual@$\overline{\mathscr{E}}^*$} \;.
\end{equation} 
Then $\overline{\mathscr{E}}^*$ is an infinite dimensional flat vector bundle over $M$. 
We have the identification 
\begin{equation}
\Omega^\cdot(M,\overline{\mathscr{E}}^*) = \Omega^{(\cdot,n,\cdot)}(\mathcal{N},\overline{E}^*) \;.
\end{equation} 

Let $\n^{\overline{\mathscr{E}}^*}$ be the restriction of $d_M^{\overline{E}^*}$ to $\Omega^\cdot(M,\overline{\mathscr{E}}^*)$. 
Then  $\n^{\overline{\mathscr{E}}^*}$ is the canonical flat connection on $\overline{\mathscr{E}}^*$.  
Set
\begin{equation}
\label{dilab-eq-decomposition-superconnection-infi-dim-dual}
A^{\overline{\mathscr{E}}^*}  = \partial^{\overline{E}^*}_N + \n^{\overline{\mathscr{E}}^*} \index{AEmathscr dual@$A^{\overline{\mathscr{E}}^*}$}\;,
\end{equation}
which acts on $\Omega^\cdot(M,\overline{\mathscr{E}}^*)$. 
Then $A^{\overline{\mathscr{E}}^*}$ is a superconnection on $\overline{\mathscr{E}}^*$.

Recall that the operator ${A^{\overline{E}^*}}'$ on $\Omega^\cdot(\mathcal{N},\overline{E}^*)$ 
was defined in \eqref{dilab-eq-def-flat-superconnection-antidual}. 
We have
\begin{equation}
A^{\overline{\mathscr{E}}^*}  = {A^{\overline{E}^*}}'\big|_{\Omega^{(\cdot,n,\cdot)}(\mathcal{N},\overline{E}^*)} \;.
\end{equation}
Then, by \eqref{dilab-eq-AEdualsquare=0}, we have
\begin{equation}
\label{dilab-eq-Adualsquare=0}
A^{\overline{\mathscr{E}}^*,2} =0 \;.
\end{equation}

Let
\begin{equation}
(\cdot,\cdot)_E : \overline{E}^* \times E \rightarrow \mathbb{C}
\end{equation}
be the canonical sesquilinear form, which extends to
\begin{equation}
(\cdot,\cdot)_E : \big( \Lambda^p(T^*N)\otimes\Lambda^n(\overline{T^*N})\otimes \overline{E}^* \big) \times \big( \Lambda^q(\overline{T^*N})\otimes E \big) \rightarrow \Lambda^{p+q}(T^*N)\otimes\Lambda^n(\overline{T^*N}) \;.
\end{equation}
We define 
\begin{align}
(\cdot,\cdot)_\mathscr{E} : \overline{\mathscr{E}}^* \times \mathscr{E} & \rightarrow \mathbb{C} \nonumber\\
(\alpha,\beta) & \mapsto \int_N (\alpha,\beta)_E \;.
\end{align}
Thus $\overline{\mathscr{E}}^*$ is formally the anti-dual of $\mathscr{E}$. 
For $\alpha\in\Omega^\cdot(M,\overline{\mathscr{E}}^*)$ and $\beta\in\Omega^\cdot(M,\mathscr{E})$, the following identities hold 
\begin{align}
\label{dilab-eq-dualrelation-connectoin-infi-dim}
\begin{split}
(\partial^{\overline{E}^*}_N\alpha,\beta)_\mathscr{E} + (-1)^{\deg\alpha}(\alpha,\overline{\partial}^E_N\beta)_\mathscr{E} & = 0 \;,\\
(\n^{\overline{\mathscr{E}}^*}\alpha,\beta)_\mathscr{E} + (-1)^{\deg\alpha}(\alpha,\n^\mathscr{E}\beta)_\mathscr{E} & = d_M(\alpha,\beta)_\mathscr{E} \;.
\end{split}
\end{align}
By (\ref{dilab-eq-decomposition-superconnection-infi-dim}), (\ref{dilab-eq-decomposition-superconnection-infi-dim-dual}) and (\ref{dilab-eq-dualrelation-connectoin-infi-dim}), we get
\begin{equation}
(A^{\overline{\mathscr{E}}^*}\alpha,\beta)_\mathscr{E} + (-1)^{\deg\alpha}(\alpha,A^\mathscr{E}\beta)_\mathscr{E} = d_M(\alpha,\beta)_\mathscr{E} \;,
\end{equation}
i.e., $A^{\overline{\mathscr{E}}^*}$ is the dual superconnection of $A^\mathscr{E}$ in the sense of \cite[Definition 1.5]{bl}.

\subsection{Hermitian metrics and connections on $TN$ and $E$}
\label{dilab-subsec-vertical-metrics}

\

From now on, we will assume that $N$ is a K\"{a}hler manifold.

Let $J : T_\mathbb{R}N\rightarrow T_\mathbb{R}N \index{J@$J$}$ be the complex structure of $N$.

\begin{prop}
There exists a fiberwise K{\"a}hler metric $g^{TN}$ on $TN$, 
i.e., a Hermitian metric on $TN$ whose restriction to each fiber $N$ is a K{\"a}hler metric.
\end{prop}
\begin{proof} 
Let $(U_i)$ be a locally finite cover of $M$ by open balls. Let $(f_i : U_i \rightarrow \mathbb{R})$ be a partition of unity.
For each $U_i$, we have the trivialization 
$\varphi_i : q^{-1}(U_i) \rightarrow N \times U_i$ 
as flat fibration. 
Let $p_i : N \times U_i \rightarrow N$ 
be the canonical projection.
Let $g^{TN}_0$ be a K{\"a}hler metric on $TN_0$.
Set
\begin{equation}
g^{TN} = \sum_i \big(q^*f_i\big)\big(\varphi_i^* p_i^* g^{TN}_0\big) \;.
\end{equation}
Then $g^{TN}$ satisfies the desired properties.
\end{proof}

Let $g^{TN}\index{gTN@$g^{TN}$}$ be a fiberwise K{\"a}hler metric on $TN$.
Let
\begin{equation}
\omega\in\smooth\big(\mathcal{N},T^*N\otimes\overline{T^*N}\big) \index{omega@$\omega$}
\end{equation} 
be the associated fiberwise K{\"a}hler form. 
Let 
\begin{equation}
dv_N = \frac{\omega^n}{n!} \in\smooth\big(\mathcal{N},\Lambda^{2n}(T^*_\mathbb{R}N)\big)
\end{equation}
be the induced fiberwise volume form.

Let $g^{\overline{TN}}$ and $g^{\Lambda^\cdot(\overline{T^*N})}$  
be the Hermitian metrics on $\overline{TN}$ and $\Lambda^\cdot(\overline{T^*N})$
induced by $g^{TN}$.

Let $g^{T_\mathbb{R}N}$ be the Riemannian metric on $T_\mathbb{R}N$ induced by $g^{TN}$.

Let $\nabla^{T_\mathbb{R}N}\index{nablaTNR@$\nabla^{T_\mathbb{R}N}$}$ be the conection on $T_{\mathbb R}N$ 
associated with $\big(g^{T_{\mathbb R}N},T^{H}\mathcal{N}\big)$ 
in the same way as in \textsection \ref{dilab-subsec-fibration-connection-metric}.
Recall that the connection $A^{TN}$ on $TN$ is defined by \eqref{dilab-eq-def-superconnection}. 
In the sequel, we change the notation as follows
\begin{equation}
\n^{TN} =  A^{TN} \index{nablaTN@$\n^{TN}$}\;.
\end{equation}
Since the metric $g^{TN}$ is fiberwise K\"{a}hler, 
the connection on $T_{\R}N$ induced by $\n^{TN}$ along the fibre $N$ coincides with $\nabla^{T_{\R}N}$. 
Moreover, the complex structure of $T_{\R}N$ is flat with respect to the flat connection on $\mathcal{N}$.  
By \eqref{eq:corr1}, \eqref{eq:corr2} and \eqref{dilab-eq-def-omegaE}, 
these two connections also coincide in horizontal directions. 
The conclusion is that 
the connection $\n^{T_{\mathbb R}N}$ preserves the complex structure $J$, 
and induces the connection $\n^{TN}$ on $TN$. 

Let $\n^{\overline{TN}}\index{nablaTNbar@$\n^{\overline{TN}}$}$ and $\n^{\Lambda^\cdot(\overline{T^*N})}\index{nablalambdaTN@$\n^{\Lambda^\cdot(\overline{T^*N})}$}$ 
be the connections on $\overline{TN}$ and $\Lambda^\cdot(\overline{T^*N})$
induced by $\n^{TN}$.

Let $g^E\index{gE@$g^E$}$ be a Hermitian metric of $E$. 
Let $\n^E$ be the connection on $E$ defined by \eqref{dilab-eq-def-superconnection}.

Let $g^{\Lambda^\cdot(T^*_\mathbb{C}N)}$ be the $\mathbb{C}$-bilinear form on $\Lambda^\cdot(T^*_\mathbb{C}N)$ induced by $g^{TN}$. Let
\begin{equation}
* : \Lambda^\cdot(T^*_\mathbb{C}N) \rightarrow \Lambda^{2n-\cdot}(T^*_\mathbb{C}N) \index{star@$*$}
\end{equation} 
be the usual Hodge operator acting on $\Lambda^\cdot(T^*_\mathbb{C}N)$, 
i.e., for $\alpha, \beta \in \Lambda^\cdot(T^*_\mathbb{C}N)$,
\begin{equation*}
g^{\Lambda^\cdot(T^*_\mathbb{C}N)}(\alpha,\beta) dv_N = \alpha \wedge \overline{* \beta} \;. 
\end{equation*}
In particular, $*$ maps $\Lambda^\cdot(\overline{T^*N})$ to $\Lambda^n(T^*N)\otimes\Lambda^{n-\cdot}(\overline{T^*N})$.

The Hermitian metric $g^E$ induces an identification $g^E : E \rightarrow \overline{E}^*$. 
The Hodge operator $*$ extends to
\begin{equation}
*^E : \Lambda^\cdot(\overline{T^*N}) \otimes E \rightarrow \Lambda^n (T^*N) \otimes \Lambda^{n-\cdot}(\overline{T^*N}) \otimes \overline{E}^*  \index{starE@$*^E$} \;.
\end{equation}

Let $g^\mathscr{E}$ be a Hermitian metric on $\mathscr{E}$, such that for $\alpha,\beta\in\mathscr{E}$,
\begin{equation}
g^\mathscr{E}(\alpha,\beta) 
= \frac{1}{(2\pi)^n} 
\int_N (g^{\Lambda^\cdot(\overline{T^*N})}\otimes g^E)(\alpha,\beta)dv_N 
= \frac{(-1)^{\deg\alpha \, \deg\beta}}{(2\pi)^n}(*^E\alpha,\beta)_\mathscr{E} \index{gEmathscr@$g^\mathscr{E}$}\;.
\end{equation}

Set
\begin{equation}
\label{dilab-eq-def-omega-mathscrE}
\omega^{\mathscr{E}} = \big(g^\mathscr{E}\big)^{-1}\n^{\overline{\mathscr{E}}^*}g^\mathscr{E} 
\in \mathscr{C}^\infty(M, T^*M \otimes \mathrm{End}(\mathscr{E})) \index{omegaEmathscr@$\omega^\mathscr{E}$}
\end{equation}
and
\begin{equation}
k_N =  \big(dv_N\big)^{-1} d_M dv_N
\in \mathscr{C}^\infty(\mathcal{N}, T^*M) \index{kN@$k_N$}\;.
\end{equation}
We define $\omega^{TN}$ in the same way as in \eqref{dilab-eq-def-omegaE}.
Let $\omega^{\Lambda^{\cdot}\left(\overline{T^{*}N}\right)} \index{omegaLambdaTN@$\omega^{\Lambda^{\cdot}\left(\overline{T^{*}N}\right)}$}$ be  
the induced action of $\omega^{TN}$ on 
$\Lambda^{\cdot}(\overline{T^{*}N})$. Then 
$\omega^{\Lambda^{\cdot}\left(\overline{T^{*}N}\right)}$ is just the 
horizontal variation of the metric 
$g^{\Lambda^{\cdot}\left(\overline{T^{*}N}\right)}$ on 
$\Lambda^{\cdot}\left(\overline{T^{*}N}\right)$ 
with respect to the flat connection.
We have
\begin{equation}
\label{dilab-eq-omega-big-e}
\omega^\mathscr{E} = \omega^{\Lambda^\cdot(\overline{T^*N})} + \omega^E + k_N \;.
\end{equation}

\subsection{The Levi-Civita superconnection}
\label{dilab-subsec-levi-civita-supperconnection}

\

Recall that $A^\mathscr{E}$ and $A^{\overline{\mathscr{E}}^*}$
were defined in 
\eqref{dilab-eq-decomposition-superconnection-infi-dim} 
and \eqref{dilab-eq-decomposition-superconnection-infi-dim-dual}.

Set
\begin{equation}
A^{\mathscr{E},*} = (*^E)^{-1}A^{\overline{\mathscr{E}}^*}*^E \index{AEmathscradjoint@$A^{\mathscr{E},*}$}\;,
\end{equation}
which acts on $\Omega^\cdot(M,\mathscr{E})$. 
Then $A^{\mathscr{E},*}$ is the adjoint superconnection of $A^\mathscr{E}$ (with respect to $g^\mathscr{E}$) in the sense of \cite[Definition 1.6]{bl}.

By \eqref{dilab-eq-Adualsquare=0}, we have
\begin{equation}
\label{dilab-eq-Astarsquare=0}
A^{\mathscr{E},*,2} = 0 \;.
\end{equation}

Set
\begin{align}
\label{dilab-eq-def-c-d}
\begin{split}
C^\mathscr{E} = \; & 2^{-N^{\Lambda^\cdot(T^*M)}} \big(A^{\mathscr{E},*} + A^\mathscr{E}\big) 2^{N^{\Lambda^\cdot(T^*M)}} \index{CEmathscr@$C^\mathscr{E}$}\;,\\
D^\mathscr{E} = \; & 2^{-N^{\Lambda^\cdot(T^*M)}} \big(A^{\mathscr{E},*} - A^\mathscr{E}\big) 2^{N^{\Lambda^\cdot(T^*M)}} \index{DEmathscr@$D^\mathscr{E}$}\;.
\end{split}
\end{align}
By \eqref{dilab-eq-Asquare=0} and \eqref{dilab-eq-Astarsquare=0}, we have
\begin{equation}
\label{dilab-eq-bianchi}
C^{\mathscr{E},2} = -D^{\mathscr{E},2} = 
2^{-N^{\Lambda^\cdot(T^*M)}}
\big[A^\mathscr{E},A^{\mathscr{E},*}\big] 
2^{N^{\Lambda^\cdot(T^*M)}}\;, \hspace{5mm}
\big[C^\mathscr{E},D^\mathscr{E}\big] = 0 \;.
\end{equation}

Let $\overline{\partial}^{E,*}_N$ be the formal adjoint of $\overline{\partial}^E_N$ with respect to $g^\mathscr{E}$. Set
\begin{equation}
D^E_N = {\overline{\partial}}^E_N + {\overline{\partial}}^{E,*}_N \index{DEN@$D^E_N$} \;, 
\end{equation}
which acts on $\mathscr{E}$.
Then $D^E_N$ is the fiberwise $\text{spin}^c$-Dirac operator associated with $g^{TN}/2$.

We recall that $\n^\mathscr{E}$ is defined in \textsection\ref{dilab-subsection-flat-superconnection}. 
Let $\n^{\mathscr{E},*}\index{nablaEmathscradjoint@$\n^{\mathscr{E},*}$}$ be the adjoint connection. 
Then
\begin{equation}
\n^{\mathscr{E},*} = \n^\mathscr{E} + \omega^\mathscr{E} \;.
\end{equation}
Set
\begin{equation}
\label{dilab-eq-def-nEu-nE}
\n^{\mathscr{E},\mathrm{u}} = \frac{1}{2}\big(\n^{\mathscr{E},*}+\n^\mathscr{E}\big)
= \n^\mathscr{E} + \frac{1}{2}\omega^\mathscr{E} \index{nablaEmathscru@$\n^{\mathscr{E},\mathrm{u}}$}\;,
\end{equation}
which is a unitary connection on $\mathscr{E}$.

We have
\begin{equation}
\label{dilab-eq-lc-superconnection}
C^\mathscr{E} = D^E_N + \n^{\mathscr{E},\mathrm{u}} \;,\hspace{5mm}
D^\mathscr{E} = \overline{\partial}^{E,*}_N - \overline{\partial}^E_N + \frac{1}{2}\omega^\mathscr{E} \;.
\end{equation}

Recall that the Levi-Civita superconnection was introduced in 
\cite{b}.
\begin{prop}
\label{dilab-prop-lc-superconnection}
The superconnection $C^\mathscr{E}$ is the Levi-Civita superconnection associated with 
$\big(T^H\mathcal{N}, g^{T_{\mathbb R}N}, g^E\big)$.
\end{prop}
\begin{proof}
Since the metric $g^{TN}$ is fibrewise K\"{a}hler, 
up to the constant $\sqrt{2}$, 
the operator  $D^E_N$ is 
the standard $\text{spin}^c$-Dirac operator along the fiber $N$. 
As we saw in \textsection \ref{dilab-subsec-vertical-metrics}, 
the connection $\n^{T_{\R}N}$ 
induced by $\n^{TN}$ is exactly the connection that was considered in 
\cite{b}. Finally, since our fibration is flat, the term in the 
Levi-Civita superconnection that contains the curvature of our 
fibration vanishes identically.
This completes the proof. 
\end{proof}

For $t>0$, 
let $C^\mathscr{E}_t,D^\mathscr{E}_t\index{CEmathscrt@$C^\mathscr{E}_t$}\index{DEmathscrt@$D^\mathscr{E}_t$}$ be $C^\mathscr{E},D^\mathscr{E}$ associated with the rescaled metric 
$g^{TN}/t$. 
By \eqref{dilab-eq-lc-superconnection}, we have
\begin{equation}
\label{dilab-eq-lc-superconnection-rescaling}
C^\mathscr{E}_t =  t\overline{\partial}^{E,*}_N + \overline{\partial}^E_N + \n^{\mathscr{E},\mathrm{u}} \;,\hspace{5mm}
D^\mathscr{E}_t =  t\overline{\partial}^{E,*}_N - \overline{\partial}^E_N + \frac{1}{2}\omega^\mathscr{E} \;.
\end{equation}

\subsection{The index bundle and its characteristic classes}
\label{dilab-subsec-index-bundle}

\

Let $H^\cdot(N,E_0)$ be the Dolbeault cohomology of 
$E_{0}$.
The action of $G$ on $E_0$ induces an action of $G$ on $H^\cdot(N,E_0)$. Set
\begin{equation}
H^\cdot(N,E) = P_G \times_G H^\cdot(N,E_0) \index{HNE@$H^\cdot(N,E)$}\;.
\end{equation}

Let $\nabla^{H^\cdot(N,E)} \index{nablaHNE@$\nabla^{H^\cdot(N,E)}$}$ be the flat connection on $H^\cdot(N,E)$ induced by the flat connection on $P_G$.
For $s\in\mathscr{C}^\infty(M,\mathscr{E})$ satisfying 
$\overline{\partial}^E_N s  = 0$,  
let 
\begin{equation}
[s]\in\mathscr{C}^\infty(M,H^\cdot(N,E))
\end{equation} 
be the corresponding cohomology class. Then
\begin{equation}
\label{dilab-def-flat-nabla-H}
\nabla^{H^\cdot(N,E)}[s] = [\n^\mathscr{E} s] \in \Omega^1(M,H^\cdot(N,E))\;.
\end{equation}

By Hodge theory, there is a canonical identification
\begin{equation}
\label{dilab-eq-hodge-id}
H^\cdot(N,E) \simeq \mathrm{ker} D^E_N \subseteq \mathscr{E} \;.
\end{equation}
Let $g^{H^\cdot(N,E)}$ be the metric on $H^\cdot(N,E)$ induced by 
$g^\mathscr{E}$ via the identification \eqref{dilab-eq-hodge-id}.

Let $\nabla^{H^\cdot(N,E),*} \index{nablaHNEad@$\nabla^{H^\cdot(N,E),*}$}$ be the adjoint connection of $\nabla^{H^\cdot(N,E)}$ with respect to $g^{H^\cdot(N,E)}$. 
Set
\begin{align}
\begin{split}
\nabla^{H^\cdot(N,E),\mathrm{u}} = \; & \frac{1}{2}\big(\nabla^{H^\cdot(N,E),*} + \nabla^{H^\cdot(N,E)} \big)  \index{nablaHNEu@$\nabla^{H^\cdot(N,E),\mathrm{u}}$} \;, \\
\omega^{H^\cdot(N,E)} = \; & \nabla^{H^\cdot(N,E),*} - \nabla^{H^\cdot(N,E)}   \index{omegaHNE@$\omega^{H^\cdot(N,E)}$} \;.
\end{split}
\end{align}
Then $\nabla^{H^\cdot(N,E),\mathrm{u}}$ is a unitary connection 
and $\omega^{H^\cdot(N,E)} \in \smooth\big(M,\mathrm{End}(H^\cdot(N,E))\big)$.

Put
\begin{equation}
\chi(N,E) = \sum_p (-1)^p \dim H^p(N,E) \index{chiNE@$\chi(N,E)$}\;.
\end{equation}

\begin{prop}
\label{dilab-prop-chern-index-theorem}
For $t>0$, we have
\begin{equation}
\label{dilab-eq-prop-chern-index-theorem}
\varphi\trs\big[\exp(D^{\mathscr{E},2}_t)\big] = \chi(N,E) \;.
\end{equation}
\end{prop}
\begin{proof}
By the  local families index theorem \cite{b}, as $t\rightarrow 0$,
\begin{equation}
\label{dilab-eq-1-proof-prop-chern-index-theorem}
\varphi\trs\big[\exp(D^{\mathscr{E},2}_t)\big] = 
q_*\big[\mathrm{Td}(TN, \nabla^{TN})\mathrm{ch}(E, \nabla^E)\big] + \mathscr{O}(\sqrt{t}) \;.
\end{equation}
Furthermore,
\begin{align}
\label{dilab-eq-2-proof-prop-chern-index-theorem}
\begin{split}
\frac{\partial}{\partial t}\trs\big[\exp(D^{\mathscr{E},2}_t)\big] 
= \; & \trs\big[\big[D^\mathscr{E}_t,\frac{\partial}{\partial t}D^\mathscr{E}_t\big]\exp(D^{\mathscr{E},2}_t)\big] \\
= \; & \trs\big[\big[D^\mathscr{E}_t,(\frac{\partial}{\partial t}D^\mathscr{E}_t)\exp(D^{\mathscr{E},2}_t)\big]\big] = 0 \;.
\end{split}
\end{align}
By Proposition \ref{dilab-thm-gen-flat2trivialclass} 
and the Riemann-Roch-Hirzebruch formula, 
we have
\begin{equation}
\label{dilab-eq-3-proof-prop-chern-index-theorem}
q_*\big[\mathrm{Td}(TN, \nabla^{TN})\mathrm{ch}(E, \nabla^E)\big] = \chi(N,E) \;.
\end{equation}
Then \eqref{dilab-eq-prop-chern-index-theorem} follows from 
\eqref{dilab-eq-1-proof-prop-chern-index-theorem}-\eqref{dilab-eq-3-proof-prop-chern-index-theorem}.
\end{proof}

\subsection{A Riemann-Roch-Grothendieck formula}
\label{dilab-subsec-riemann-roch-grothendieck}

\

For  $t>0$, set
\begin{align}
\label{dilab-eq-def-alphatbetat}
\begin{split}
\alpha_t & = \sqrt{2\pi i}\varphi\trs\Big[D^\mathscr{E}_t\exp(D^{\mathscr{E},2}_t)\Big] 
\in\Omega^\mathrm{odd}(M)
\index{a lphat@$\alpha_t$}\;,\\
\beta_t & = \varphi\trs\Big[ \frac{N^{\Lambda^\cdot(\overline{T^*N})}}{2}(1+2D^{\mathscr{E},2}_t) \exp(D^{\mathscr{E},2}_t) \Big] 
\in\Omega^\mathrm{even}(M)
\index{b etat@$\beta_t$}\;.
\end{split}
\end{align}

\begin{prop}
\label{dilab-prop-constant-class}
For  $t>0$, the differential form $\alpha_t$ is closed. 
Its cohomology class is independent of  $g^{TN}$, $g^{E}$ and $t$.
\end{prop}
\begin{proof}
By \eqref{dilab-eq-bianchi}, we have
\begin{equation}
d_M\sqrt{2\pi i}\varphi\trs\big[D^\mathscr{E}_t\exp(D^{\mathscr{E},2}_t)\big] = 
\varphi\trs\big[\big[C^\mathscr{E}_t,D^\mathscr{E}_t\exp(D^{\mathscr{E},2}_t)\big]\big] = 0 \;,
\end{equation}
which proves the closeness. 
Then, by the functoriality of our constructions, 
$[\alpha_t]\in H^\cdot(M)$ is independebt of the metrics. 
In particular, it is independent of $t$.
\end{proof}

\begin{prop}
\label{dilab-prop-trans-alpha-beta}
For $t>0$, the following identity holds:
\begin{equation}
\label{dilab-eq-prop-trans-alpha-beta}
\frac{\partial}{\partial t}\alpha_t = \frac{1}{t}d_M \beta_t \;.
\end{equation}
\end{prop}
\begin{proof}
Set
\begin{equation}
\mathcal{N}_+ = \mathcal{N}\times\mathbb{R}_+ \;,\hspace{5mm} M_+ = M\times\mathbb{R}_+ \index{Nmathcal+@$\mathcal{N}_+$}\index{M+@$M_+$}\;.
\end{equation}
Let
\begin{equation}
q_+ = q \oplus \mathrm{id}_{\mathbb{R}_+} : \mathcal{N}_+\rightarrow M_+
\end{equation}
be the obvious projection. 
Let $t$ be the coordinate on $\mathbb{R}_+$. 

We equip $TN$ with the metric $\frac{1}{t}g^{TN}$. Let $\mathscr{E}_+\index{Emathscr+@$\mathscr{E}_+$}$, 
$\omega^{\mathscr{E}_+}$, $C^{\mathscr{E}_+}$, $D^{\mathscr{E}_+}$ be the corresponding objects associated with the new fibration. 
The following identities hold 
(cf. \eqref{dilab-eq-def-omega-mathscrE})
\begin{align}
\label{dilab-proof-prop-trans-alpha-beta-eq-0}
\begin{split}
d_{M_+} & = d_M + dt\wedge\frac{\partial}{\partial t} \;,\\
\omega^{\mathscr{E}_+} & = \omega^{\mathscr{E}} + \frac{1}{t}dt\wedge\big(N^{\Lambda^\cdot(\overline{T^*N})}-n\big) \;.
\end{split}
\end{align}
Then, by \eqref{dilab-eq-lc-superconnection} and \eqref{dilab-eq-lc-superconnection-rescaling}, we get
\begin{align}
\begin{split}
C^{\mathscr{E}_+} & = C^{\mathscr{E}}_t + dt\wedge\frac{\partial}{\partial t} + \frac{1}{2t}dt\wedge\big(N^{\Lambda^\cdot(\overline{T^*N})}-n\big) \;,\\
D^{\mathscr{E}_+} & = D^{\mathscr{E}}_t + \frac{1}{2t}dt\wedge\big(N^{\Lambda^\cdot(\overline{T^*N})}-n\big) \;.
\end{split}
\end{align}
Thus
\begin{align}
\label{dilab-proof-prop-trans-alpha-beta-eq-1}
\begin{split}
& \sqrt{2\pi i}\varphi\trs\big[D^{\mathscr{E}_+}\exp(D^{\mathscr{E}_+,2})\big] \\
= \; & \sqrt{2\pi i}\varphi\trs\big[D^{\mathscr{E}}\exp(D^{\mathscr{E},2})\big] 
+ \frac{1}{2t}dt\wedge\varphi\trs\big[\big(N^{\Lambda^\cdot(\overline{T^*N})}-n\big)\exp(D^{\mathscr{E},2})\big] \\
& + \sqrt{2\pi i}\varphi\trs\Big[D^{\mathscr{E}}\exp\big(D^{\mathscr{E},2}
+\big[D^{\mathscr{E}},\frac{1}{2t}dt\wedge N^{\Lambda^\cdot(\overline{T^*N})}\big]\big)\Big] \\
= \; & \alpha_t  + \frac{1}{2t}dt\wedge\varphi\trs\big[N^{\Lambda^\cdot(\overline{T^*N})}\exp(D^{\mathscr{E},2})\big] - \chi(N,E)\frac{n}{2t}dt \\
& + \sqrt{2\pi i}\varphi\trs\Big[D^{\mathscr{E}}\big[D^{\mathscr{E}},
\exp\big(D^{\mathscr{E},2}+ \frac{1}{2t}dt\wedge N^{\Lambda^\cdot(\overline{T^*N})}\big)\big]\Big] \\
= \; & \alpha_t  + \frac{1}{2t}dt\wedge\varphi\trs\big[N^{\Lambda^\cdot(\overline{T^*N})}\exp(D^{\mathscr{E},2})\big] - \chi(N,E)\frac{n}{2t}dt \\
& + \sqrt{2\pi i}\varphi\trs\Big[\big[D^{\mathscr{E}},D^{\mathscr{E}}\big]
\exp\big(D^{\mathscr{E},2}+ \frac{1}{2t}dt\wedge N^{\Lambda^\cdot(\overline{T^*N})}\big)\Big] \\
= \; & \alpha_t +\frac{1}{2t}dt\wedge\beta_t - \chi(N,E)\frac{n}{2t}dt \in \Omega^\cdot(M_+) \;.
\end{split}
\end{align}

By Proposition \ref{dilab-prop-constant-class}, we have
\begin{equation}
\label{dilab-proof-prop-trans-alpha-beta-eq-2}
d_{M_+}\sqrt{2\pi i}\varphi\trs\big[D^{\mathscr{E}_+}\exp(D^{\mathscr{E}_+,2})\big] = 0 \;.
\end{equation}

By \eqref{dilab-proof-prop-trans-alpha-beta-eq-0}, 
\eqref{dilab-proof-prop-trans-alpha-beta-eq-1} and 
\eqref{dilab-proof-prop-trans-alpha-beta-eq-2}, 
we get \eqref{dilab-eq-prop-trans-alpha-beta}.
\end{proof}

Set $f(x)=xe^{x^2}$.
Following \cite[Definition 1.7]{bl}, 
we define the odd characteristic form
\begin{equation}
f(H^\cdot(N,E), \nabla^{H^\cdot(N,E)}, g^{H^\cdot(N,E)}) = 
\sqrt{2\pi i}\varphi\trs\big[ f(\omega^{H^\cdot(N,E)}/2) \big] 
\in\Omega^\mathrm{odd}(M) \;.
\end{equation}

Put
\begin{equation}
\chi'(N,E) = \sum_p(-1)^pp\dim H^p(N,E) \index{chiNEprim@$\chi'(N,E)$}\;.
\end{equation}

Now we state the central result in this section. 
Its proof will be delayed to \textsection \ref{dilab-subsec-intermediate-rrg}.

\begin{thm}
\label{dilab-prop-large-small-time-convergence}
As $t\rightarrow + \infty$,
\begin{align}
\begin{split}
\label{dilab-eq-large-time-convergence}
\alpha_t & = f(H^\cdot(N,E), \nabla^{H^\cdot(N,E)}, g^{H^\cdot(N,E)})+\mathscr{O}\big(\frac{1}{\sqrt{t}}\big) \;,\\
\beta_t & = \frac{1}{2}{\chi}'(N,E) + \mathscr{O}\big(\frac{1}{\sqrt{t}}\big) \;.
\end{split}
\end{align}

As $t\rightarrow 0$,
\begin{align}
\begin{split}
\label{dilab-eq-small-time-convergence}
\alpha_t = \; & q_*\Big[\widetilde{\mathrm{Td}}(TN, g^{TN})*\widetilde{\mathrm{ch}}(E, g^E) \Big]  \\
\; & + \frac{1}{2t} d_M q_*\Big[\frac{\omega}{2\pi}\mathrm{Td}(TN, \nabla^{TN})\mathrm{ch}(E, \nabla^E)\Big] + \mathscr{O}\big(\sqrt{t}\big) \;,\\
\beta_t = \; & - \frac{1}{2}q_*\Big[\mathrm{Td}'(TN, \nabla^{TN})\mathrm{ch}(E, \nabla^E)\Big] + \frac{n}{2}\chi(N,E) \\
\; & - \frac{1}{2t} q_*\Big[\frac{\omega}{2\pi}\mathrm{Td}(TN, \nabla^{TN})\mathrm{ch}(E, \nabla^E)\Big] + \mathscr{O}\big(\sqrt{t}\big) \;.
\end{split}
\end{align}
\end{thm}

\begin{rem}
\label{dilab-rem-small-time-convergence}
By Proposition \ref{dilab-prop-annulation-curvature}, we have
\begin{equation}
q_*\Big[\frac{\omega}{2\pi}\mathrm{Td}(TN, \nabla^{TN})\mathrm{ch}(E, \nabla^E)\Big] 
\in \smooth(M) \;.
\end{equation}
\end{rem}

By Proposition \ref{dilab-prop-constant-class} 
and Theorem \ref{dilab-prop-large-small-time-convergence},
we get the following R.R.G. formula.

\begin{thm}
\label{dilab-thm-riemann-roch-grothendieck}
We have
\begin{align}
\label{dilab-thm-riemann-roch-grothendieck-eq}
\begin{split}
 & \Big[f(H^\cdot(N,E), \nabla^{H^\cdot(N,E)}, g^{H^\cdot(N,E)})\Big] \\
= \; & \Big[q_*\Big[ \widetilde{\mathrm{Td}}(TN, g^{TN})*\widetilde{\mathrm{ch}}(E, g^E) \Big]\Big] \in H^\mathrm{odd}(M,\mathbb{R}) \;.
\end{split}
\end{align}
\end{thm}

\subsection{Several intermediate results and the proof of Theorem \ref{dilab-prop-large-small-time-convergence}}
\label{dilab-subsec-intermediate-rrg}

\

We will now introduce various new odd Grassmann variables in order to 
be able to compute exactly the asymptotics of certain superconnection 
forms as $t\to 0$, and also to overcome the divergence of certain 
expressions. Our methods are closely related to the methods of 
\cite{bgs2,bgs3,bk}, 
where similar difficulties also appeared.

Let $a\index{a@$a$}$ be an additional complex coordinate, $\epsilon\index{epsilon@$\epsilon$}$ be an auxiliary odd Grassmann variable.

For 
\begin{equation}
u,v\in \Big\{ 
1 \,,\, da \,,\, d\bar{a} \,,\, dad\bar{a} \,,\,
\epsilon \,,\, \epsilon da \,,\, \epsilon d\bar{a} \,,\, \epsilon dad\bar{a} \Big\} 
\end{equation}  
and $\sigma\in\Omega^\cdot(M)$,  
we denote
\begin{align}
\begin{split}
(v \wedge \sigma)^u = 
\left\{
\begin{array}{rl}
\sigma & \text{ if } u = v \;,\\
     0 & \text{ else } \;.
\end{array} \right.
\end{split}
\end{align}

\begin{lemme}
\label{dilab-prop-a-bar-a-formula}
The following identity holds
\begin{align}
\label{dilab-eq-1-prop-a-bar-a-formula}
\begin{split}
& \trs \Big[D^\mathscr{E}\exp\big(D^{\mathscr{E},2}\big)\Big] \\
= \; & \trs \Big[\exp\big(-C^{\mathscr{E},2} - da\,\frac{1}{2}\big(\overline{\partial}^E_N+\overline{\partial}^{E,*}_N\big) \\
& \hspace{20mm} - d\bar{a}\,\big[\overline{\partial}^E_N+\overline{\partial}^{E,*}_N,\frac{\epsilon}{2}\omega^\mathscr{E}\big] + dad\bar{a}\,\frac{\epsilon}{2}\omega^\mathscr{E}\big)\Big]^{\epsilon dad\bar{a}} \\
& + d_M \trs \Big[\frac{1}{2}N^{\Lambda^\cdot(\overline{T^*N})}\exp\big( D^{\mathscr{E},2} \big)\Big] \;.
\end{split}
\end{align}
\end{lemme}
\begin{proof}
By \eqref{dilab-eq-bianchi} and \eqref{dilab-eq-lc-superconnection}, we have
\begin{align}
\label{dilab-proof-prop-a-bar-a-formula-eq-2}
\begin{split}
\big[N^{\Lambda^\cdot(\overline{T^*N})},C^{\mathscr{E},2}\big]
= \; & -\big[N^{\Lambda^\cdot(\overline{T^*N})},D^{\mathscr{E},2}\big] \\
= \; & -\big[N^{\Lambda^\cdot(\overline{T^*N})},\big[\overline{\partial}^{E,*}_N-\overline{\partial}^E_N,\frac{1}{2}\omega^\mathscr{E}\big]\big] 
= \big[\overline{\partial}^E_N+\overline{\partial}^{E,*}_N,\frac{1}{2}\omega^\mathscr{E}\big] \;,
\end{split}
\end{align}
which implies
\begin{align}
\label{dilab-proof-prop-a-bar-a-formula-eq-3}
\begin{split}
& \trs \Big[\exp\big(-C^{\mathscr{E},2} 
- da\,\frac{1}{2}\big(\overline{\partial}^E_N+\overline{\partial}^{E,*}_N\big)
- d\bar{a}\,\big[\overline{\partial}^E_N+\overline{\partial}^{E,*}_N,\frac{\epsilon}{2}\omega^\mathscr{E}\big] 
\big)\Big]^{\epsilon dad\bar{a}} \\
= \; & \frac{\partial}{\partial b} \trs \Big[ - \frac{1}{2}(\overline{\partial}^E_N+\overline{\partial}^{E,*}_N) \exp\big( -C^{\mathscr{E},2} + b\big[\overline{\partial}^E_N+\overline{\partial}^{E,*}_N,\frac{1}{2}\omega^\mathscr{E}\big] \big)\Big]_{b=0} \\
= \; & \frac{\partial}{\partial b} \trs \Big[ - \frac{1}{2}(\overline{\partial}^E_N+\overline{\partial}^{E,*}_N) \exp\big( -C^{\mathscr{E},2} + b\big[N^{\Lambda^\cdot(\overline{T^*N})},C^{\mathscr{E},2}\big] \big)\Big]_{b=0} \\
= \; & \frac{\partial}{\partial b} \trs \Big[ - \frac{1}{2}(\overline{\partial}^E_N+\overline{\partial}^{E,*}_N) \big[N^{\Lambda^\cdot(\overline{T^*N})},\exp\big( -C^{\mathscr{E},2} \big)\big]\Big]\\
= \; & \trs \Big[ -\frac{1}{2}\big[N^{\Lambda^\cdot(\overline{T^*N})},\overline{\partial}^E_N+\overline{\partial}^{E,*}_N\big]\exp\big( - C^{\mathscr{E},2} \big)\Big] \\
= \; & \trs \Big[\frac{1}{2}\big(\overline{\partial}^{E,*}_N-\overline{\partial}^E_N\big)\exp\big( D^{\mathscr{E},2} \big)\Big] \;.
\end{split}
\end{align}
Then
\begin{align}
\begin{split}
& \trs \Big[\exp\big(-C^{\mathscr{E},2} 
- da\,\frac{1}{2}\big(\overline{\partial}^E_N+\overline{\partial}^{E,*}_N\big) \\
&\hspace{35mm} - d\bar{a}\,\big[\overline{\partial}^E_N+\overline{\partial}^{E,*}_N,\frac{\epsilon}{2}\omega^\mathscr{E}\big] 
+ dad\bar{a}\,\frac{\epsilon}{2}\omega^\mathscr{E}\big)\Big]^{\epsilon dad\bar{a}} \\
= \; & \trs \Big[\exp\big(-C^{\mathscr{E},2} 
- da\,\frac{1}{2}\big(\overline{\partial}^E_N+\overline{\partial}^{E,*}_N\big) 
- d\bar{a}\,\big[\overline{\partial}^E_N+\overline{\partial}^{E,*}_N,\frac{\epsilon}{2}\omega^\mathscr{E}\big] \big)\Big]^{\epsilon dad\bar{a}} \\
& + \trs \Big[\frac{1}{2}\omega^\mathscr{E}\exp\big(D^{\mathscr{E},2}\big)\Big]  \\
= \; & \trs \Big[\frac{1}{2}\big(\overline{\partial}^{E,*}_N - \overline{\partial}^E_N + \omega^\mathscr{E}\big)\exp\big(D^{\mathscr{E},2}\big)\Big]  \\
= \; & \trs \Big[\big(\overline{\partial}^{E,*}_N - \overline{\partial}^E_N + \frac{1}{2}\omega^\mathscr{E}\big)\exp\big(D^{\mathscr{E},2}\big)\Big]  
-\trs \Big[\frac{1}{2} \big(\overline{\partial}^{E,*}_N - \overline{\partial}^E_N \big)\exp\big(D^{\mathscr{E},2}\big)\Big] \\
= \; & \trs \Big[D^\mathscr{E}\exp\big(D^{\mathscr{E},2}\big)\Big] 
- \trs \Big[\big[C^\mathscr{E},\frac{1}{2}N^{\Lambda^\cdot(\overline{T^*N})}\big]\exp\big( D^{\mathscr{E},2} \big)\Big] \\
= \; & \trs \Big[D^\mathscr{E}\exp\big(D^{\mathscr{E},2}\big)\Big] 
- d_M \trs \Big[\frac{1}{2}N^{\Lambda^\cdot(\overline{T^*N})}\exp\big( D^{\mathscr{E},2} \big)\Big] \;.
\end{split}
\end{align}
The last equation is just what we needed to prove. 
\end{proof}

Let $\mathcal{N}_+$, $M_+$, $q_+$, $\mathscr{E}_+$, $\omega^{\mathscr{E}_+}$, $C^{\mathscr{E}_+}$ and $D^{\mathscr{E}_+}$ 
be the same as in the proof of Proposition \ref{dilab-prop-trans-alpha-beta}.

\begin{lemme}
\label{dilab-prop-a-bar-a-formula-second}
For $t>0$, the following identity holds:
\begin{align}
\label{dilab-eq-prop-a-bar-a-formula-second}
\begin{split}
& \big(N^{\Lambda^\cdot(T^*M)} + 1 + t\frac{\partial}{\partial t} \big) 
\trs \Big[\frac{1}{2}N^{\Lambda^\cdot(\overline{T^*N})}\exp\big( D^{\mathscr{E},2}_t \big)\Big] \\
= \; & \trs \Big[\exp\big(-C^{\mathscr{E}_+,2} 
- da\,\frac{1}{2}\big(\overline{\partial}^E_N + t\overline{\partial}^{E,*}_N\big) \\
& \hspace{20mm} - d\bar{a}\,\big[\overline{\partial}^E_N + t\overline{\partial}^{E,*}_N , \frac{\epsilon t}{2}\omega^{\mathscr{E}_+}\big] 
+ dad\bar{a}\,\frac{\epsilon t}{2}\omega^{\mathscr{E}_+}\big)\Big]^{\epsilon dad\bar{a}dt} \\
& + \text{closed form} \;.
\end{split}
\end{align}
\end{lemme}
\begin{proof}
By \eqref{dilab-eq-1-prop-a-bar-a-formula}, we get
\begin{align}
\begin{split}
& \trs \Big[D^{\mathscr{E}_+}\exp\big(D^{\mathscr{E}_+,2}\big)\Big] \\
= \; & \trs \Big[\exp\big(-C^{\mathscr{E}_+,2} 
- da\,\frac{1}{2}\big( \overline{\partial}^E_N + t\overline{\partial}^{E,*}_N \big) \\
& \hspace{20mm} - d\bar{a}\,\big[\overline{\partial}^E_N + t\overline{\partial}^{E,*}_N , \frac{\epsilon}{2}\omega^{\mathscr{E}_+}\big]
+ dad\bar{a}\,\frac{\epsilon}{2}\omega^{\mathscr{E}_+}\big)\Big]^{\epsilon dad\bar{a}} \\
& + d_{M_+} \trs \Big[\frac{1}{2}N^{\Lambda^\cdot(\overline{T^*N})}\exp\big( D^{\mathscr{E}_+,2} \big)\Big] \;.
\end{split}
\end{align}
Taking the $dt$ component, we get
\begin{align}
\label{dilab-proof-prop-a-bar-a-formula-second-eq-0}
\begin{split}
& \trs \Big[\frac{1}{2t}\big(N^{\Lambda^\cdot(\overline{T^*N})}-n\big)\exp\big(D^{\mathscr{E},2}_t\big)\Big] \\
& + \trs \Big[D^{\mathscr{E}}_t\exp\big(\big( D^{\mathscr{E}}_t + dt\,\frac{1}{2t}N^{\Lambda^\cdot(\overline{T^*N})} - dt\,\frac{n}{2t}\big)^2\big)\Big]^{dt} \\
= \; & \trs \Big[\exp\big(-C^{\mathscr{E}_+,2} 
- da\,\frac{1}{2}\big( \overline{\partial}^E_N + t\overline{\partial}^{E,*}_N \big) \\
& \hspace{20mm} - d\bar{a}\,\big[\overline{\partial}^E_N + t\overline{\partial}^{E,*}_N , \frac{\epsilon}{2}\omega^{\mathscr{E}_+}\big] 
+ dad\bar{a}\,\frac{\epsilon}{2}\omega^{\mathscr{E}_+}\big)\Big]^{\epsilon dad\bar{a}dt} \\
& - d_M \trs \Big[\frac{1}{2}N^{\Lambda^\cdot(\overline{T^*N})}\exp\big( \big(D^{\mathscr{E}}_t + \frac{1}{2}dt\, N^{\Lambda^\cdot(\overline{T^*N})}\big)^2 \big)\Big]^{dt} \\
& + \frac{\partial}{\partial t} \trs \Big[\frac{1}{2}N^{\Lambda^\cdot(\overline{T^*N})}\exp\big( D^{\mathscr{E},2}_t \big)\Big] \;.
\end{split}
\end{align} 
We multiply \eqref{dilab-proof-prop-a-bar-a-formula-second-eq-0} by $t$ and subtract the closed forms. 
Since $dt$ supercommutes with $N^{\Lambda^\cdot(\overline{T^*N})}$ and $D^{\mathscr{E}}_t$, 
By Proposition \ref{dilab-prop-chern-index-theorem}, \ref{dilab-prop-constant-class},
we can delete $\frac{n}{2t}$ and $dt\,\frac{n}{2t}$ on the left-hand side of \eqref{dilab-proof-prop-a-bar-a-formula-second-eq-0}.
We obtain
\begin{align}
\label{dilab-proof-prop-a-bar-a-formula-second-eq-1}
\begin{split}
& \trs \Big[\frac{1}{2}N^{\Lambda^\cdot(\overline{T^*N})}\exp\big(D^{\mathscr{E},2}_t\big)\Big] \\
& + \trs \Big[D^{\mathscr{E}}_t\exp\big(\big( D^{\mathscr{E}}_t + dt\,\frac{1}{2}N^{\Lambda^\cdot(\overline{T^*N})} \big)^2\big)\Big]^{dt} \\
= \; & \trs \Big[\exp\big(-C^{\mathscr{E}_+,2} 
- da\,\frac{1}{2}\big( \overline{\partial}^E_N + t\overline{\partial}^{E,*}_N \big) \\
& \hspace{20mm} - d\bar{a}\,\big[\overline{\partial}^E_N + t\overline{\partial}^{E,*}_N , \frac{\epsilon t}{2}\omega^{\mathscr{E}_+}\big] 
+ dad\bar{a}\,\frac{\epsilon t}{2}\omega^{\mathscr{E}_+}\big)\Big]^{\epsilon dad\bar{a}dt} \\
& + t\frac{\partial}{\partial t} \trs \Big[\frac{1}{2}N^{\Lambda^\cdot(\overline{T^*N})}\exp\big(D^{\mathscr{E},2}_t\big)\Big] 
+ \text{closed form} \;.
\end{split}
\end{align}

We have
\begin{align}
\label{dilab-proof-prop-a-bar-a-formula-second-eq-2}
\begin{split}
& d_M \trs \Big[D^{\mathscr{E}}_t\exp\big(\big( D^{\mathscr{E}}_t + dt\,\frac{1}{2}N^{\Lambda^\cdot(\overline{T^*N})} \big)^2\big)\Big]^{dt} \\
= \; & \trs \Big[\big[C^{\mathscr{E}}_t,D^{\mathscr{E}}_t\exp\big(\big( D^{\mathscr{E}}_t + dt\,\frac{1}{2}N^{\Lambda^\cdot(\overline{T^*N})} \big)^2\big)\big]\Big]^{dt} \\
= \; & - \trs \Big[D^{\mathscr{E}}_t\exp\big(D^{\mathscr{E},2}_t 
+ \big[C^{\mathscr{E}}_t,\big[D^{\mathscr{E}}_t,dt\,\frac{1}{2}N^{\Lambda^\cdot(\overline{T^*N})}\big]\big] \big)\Big]^{dt} \\
= \; & \trs \Big[D^{\mathscr{E}}_t\exp\big(D^{\mathscr{E},2}_t 
+ \big[D^{\mathscr{E}}_t,\big[C^{\mathscr{E}}_t,dt\,\frac{1}{2}N^{\Lambda^\cdot(\overline{T^*N})}\big]\big] \big)\Big]^{dt} \\
= \; & \trs \Big[D^{\mathscr{E}}_t\big[D^{\mathscr{E}}_t,\exp\big(D^{\mathscr{E},2}_t 
+ \big[C^{\mathscr{E}}_t,dt\,\frac{1}{2}N^{\Lambda^\cdot(\overline{T^*N})}\big]\big)\big]\Big]^{dt} \\
= \; & \trs \Big[2D^{\mathscr{E},2}_t\exp\big(D^{\mathscr{E},2}_t 
+ \big[C^{\mathscr{E}}_t,dt\,\frac{1}{2}N^{\Lambda^\cdot(\overline{T^*N})}\big]\big)\Big]^{dt} \\
= \; & \left(d_M \trs \Big[2D^{\mathscr{E},2}_t\exp\big(D^{\mathscr{E},2}_t 
+  dt\,\frac{1}{2}N^{\Lambda^\cdot(\overline{T^*N})}\big)\Big] \right)^{dt} \;.
\end{split}
\end{align}
Thus
\begin{align}
\label{dilab-proof-prop-a-bar-a-formula-second-eq-3}
\begin{split}
& \trs \Big[D^{\mathscr{E}}_t\exp\big(\big( D^{\mathscr{E}}_t + dt\,\frac{1}{2}N^{\Lambda^\cdot(\overline{T^*N})} \big)^2\big)\Big]^{dt} \\
= \; & \trs \Big[2D^{\mathscr{E},2}_t\exp\big(D^{\mathscr{E},2}_t 
+ dt\,\frac{1}{2}N^{\Lambda^\cdot(\overline{T^*N})}\big)\Big]^{dt} + \text{closed form} \\
= \; & \frac{\partial}{\partial b} \trs \Big[N^{\Lambda^\cdot(\overline{T^*N})}\exp\big((1+b)D^{\mathscr{E},2}_t\big)\Big]_{b=0} + \text{closed form} \\
= \; & \frac{\partial}{\partial b} \trs \Big[N^{\Lambda^\cdot(\overline{T^*N})}\exp\Big(
(1+b)^{\frac{1}{2}N^{\Lambda^\cdot(T^*M)}} D^{\mathscr{E},2}_{(1+b)t} (1+b)^{-\frac{1}{2}N^{\Lambda^\cdot(T^*M)}} \Big)\Big]_{b=0} \\
& + \text{closed form} \\
= \; & \frac{\partial}{\partial b} (1+b)^{\frac{1}{2}N^{\Lambda^\cdot(T^*M)}} 
\trs \Big[N^{\Lambda^\cdot(\overline{T^*N})}\exp\big(D^{\mathscr{E},2}_{(1+b)t}\big)\Big]_{b=0} + \text{closed form} \\
= \; & t\frac{\partial}{\partial t} \trs \Big[N^{\Lambda^\cdot(\overline{T^*N})}\exp\big(D^{\mathscr{E},2}_{t}\big)\Big] \\
& + \frac{1}{2}N^{\Lambda^\cdot(T^*M)} \trs \Big[N^{\Lambda^\cdot(\overline{T^*N})}\exp\big(D^{\mathscr{E},2}_{t}\big)\Big] + \text{closed form} \;.
\end{split}
\end{align}

By \eqref{dilab-proof-prop-a-bar-a-formula-second-eq-1} and \eqref{dilab-proof-prop-a-bar-a-formula-second-eq-3}, 
we get \eqref{dilab-eq-prop-a-bar-a-formula-second}.
\end{proof}

Let $r^N\index{rN@$r^N$}$ be the scalar curvature of $(N,g^{TN})$. 
Let
\begin{equation}
R^E = \n^{E,2} \;,\hspace{5mm} 
R^{TN} = \n^{TN,2}
\end{equation}
be the curvatures of $\nabla^{E}$ and $\nabla^{TN}$ on $E$ and $TN$ over $\mathcal{N}$.
Let $\nabla^{\Lambda^{n}\left(TN\right)}$ be the connection $\Lambda^{n}\left(TN\right)$, 
which is induced by $\n^{TN}$.
Then its curvature is $\mathrm{Tr}\left[R^{TN}\right]$.

Recall that $S^{T_{\mathbb R}N}$ was defined in 
Definition \ref{defSTX}. Since our fibration is flat, it follows from 
\cite[(1.28)]{b} that, for $U\in T_{\mathbb R}N$ and $V,W\in 
T^{H}\mathcal{N}$,
\begin{equation}
\label{eq:van1}
\left\langle  S^{T_{\mathbb R}N}\left(U\right)V,W\right\rangle = \big\langle U,T(V,W)\big\rangle = 0 \;.
\end{equation}

Let $\n^{\Lambda^\cdot(\overline{T^*N})\otimes E} \index{nablaLambdaTNE@$\n^{\Lambda^\cdot(\overline{T^*N})\otimes E}$}$ be the connection on $\Lambda^\cdot(\overline{T^*N})\otimes E$
induced by $\n^{\Lambda^\cdot(\overline{T^*N})}$ and $E$.

Recall that $\omega$ is the fiberwise K{\"a}hler form, 
$\omega^{TN}$ and $\omega^E$ are the variation of metrics on $TN$ and $E$.
We also recall that $c(\cdot)$ is the Clifford action 
(cf. \textsection \eqref{dilab-eq-pre-clifford-complex-rep}) 
associated with $g^{TN}/2$.

Let $(e_i)_{1 \leqslant i \leqslant 2n}\index{ei@$e_i$}$ be an orthonormal basis of 
$T_\mathbb{R}N$, 
let $(e^i)_{1 \leqslant i \leqslant 2n}\index{eidual@$e^i$}$ be the corresponding dual basis. 
Let $(f_\alpha)_{1 \leqslant \alpha \leqslant m}\index{falpha@$f_\alpha$}$ a basis of $TM$. 
We identify the $f_{\alpha}$ with their horizontal lifts in
 $T^H\mathcal{N}$. 
Let $(f^\alpha)_{1 \leqslant \alpha \leqslant m}\index{falphadual@$f^\alpha$}$ be the corresponding dual basis.

To interpret properly the formula that follows, 
we need to extend the basis $e_{i}$ 
to a parallel basis of $T_{\mathbb R}N$ 
near the point $x$ which is considered. 
Moreover, we may suppose that $\n^{T_\mathbb{R}N}_\cdot e_i = 0$ at the point $x$.

\begin{prop}
\label{dilab-prop-lichnerowicz-a-bar-a}
The following identity holds:
\begin{align}
\label{dilab-eq-lichnerowicz-a-bar-a}
\begin{split}
& -C^{\mathscr{E},2} 
- da\,\frac{1}{2}(\overline{\partial}^E_N+\overline{\partial}^{E,*}_N) 
- d\bar{a}\,\big[\overline{\partial}^E_N+\overline{\partial}^{E,*}_N,\frac{\epsilon}{2}\omega^\mathscr{E}\big] 
+ dad\bar{a}\,\frac{\epsilon}{2}\omega^\mathscr{E} \\
= \; & \frac{1}{2} \Big(\n^{\Lambda^\cdot(\overline{T^*N})\otimes E}_{e_i} + \langle S^{T_\mathbb{R}N}(e_i)e_j,f_\alpha \rangle c(e_j)f^\alpha \\
& \hspace{10mm} - da\,\frac{1}{2}c(e_i) - d\bar{a}\epsilon\,\frac{\sqrt{-1}}{2}
(d_M\omega)(e_i,e_j)c(e_j) \Big)^2  \\
& - d\bar{a}\epsilon\,\big[ \n^{\Lambda^\cdot(\overline{T^*N})\otimes E}_{e_i} , \frac{1}{2}\omega^E - \frac{1}{8}(d_M\omega^{TN})(e_j,Je_j) \big] c(e_i) \\
& + dad\bar{a}\epsilon\,\Big(\frac{1}{2}\omega^E - \frac{1}{8}(d_M\omega)(e_j,Je_j)\Big) \\
& - \frac{1}{2}\Big(R^E+\frac{1}{2}\tr[R^{TN}]\Big)(e_i,e_j)c(e_i)c(e_j) 
- \Big(R^E+\frac{1}{2}\tr[R^{TN}]\Big)(e_i,f_\alpha\big)c(e_i)f^\alpha \\
& - \frac{1}{2}\Big(R^E+\frac{1}{2}\tr[R^{TN}]\Big)(f_\alpha,f_\beta)f^\alpha f^\beta 
- \frac{1}{8}r^N \;.
\end{split}
\end{align}
\end{prop}
\begin{proof}
Applying \cite[Theorem 3.5]{b} with $t=1/\sqrt{2}$ and \eqref{eq:van1}, we have
\begin{align}
\label{dilab-eq-family-lichnerowicz}
\begin{split}
  & -C^{\mathscr{E},2} \\
= \; & \frac{1}{2}\Big(\n^{\Lambda^\cdot(\overline{T^*N})\otimes E}_{e_i} + \langle S^{T_\mathbb{R}N}(e_i)e_j,f_\alpha \rangle c(e_j)f^\alpha \Big)^2 \\
& - \frac{1}{2}\big(R^E+\frac{1}{2}\tr[R^{TN}]\big)(e_i,e_j)c(e_i)c(e_j) 
- \big(R^E+\frac{1}{2}\tr[R^{TN}]\big)(e_i,f_\alpha)c(e_i)f^\alpha \\
& - \frac{1}{2}\big(R^E+\frac{1}{2}\tr[R^{TN}]\big)(f_\alpha,f_\beta)f^\alpha f^\beta - \frac{1}{8}r^N  \;.
\end{split}
\end{align}

Taking the degree $0$ part of (\ref{dilab-eq-family-lichnerowicz}), we get
\begin{align}
\label{dilab-eq-point-lichnerowicz}
\begin{split}
  & -\big(\overline{\partial}^E_N + \overline{\partial}^{E,*}_N\big)^2 \\
= & \; \frac{1}{2}\Big(\n^{\Lambda^\cdot(\overline{T^*N})\otimes E}_{e_i}\Big)^2 
- \frac{1}{2}\big(R^E+\frac{1}{2}\tr[R^{TN}]\big)(e_i,e_j)c(e_i)c(e_j) 
- \frac{1}{8}r^N \;.
\end{split}
\end{align}

By \cite[Proposition 1.19]{bgs3} and by (\ref{dilab-eq-omega-big-e}), we get
\begin{equation}
\label{dilab-eq-omega-big-e-explicite}
\omega^\mathscr{E} = -\frac{\sqrt{-1}}{2}(d_M\omega)(e_i,e_j)c(e_i)c(e_j) - \frac{1}{4}(d_M\omega)(e_i,Je_i) + \omega^E \;.
\end{equation}
Since $d_N\omega=0$ and $[d_N,d_M]=0$, 
we have $d_Nd_M\omega=0$. 
Therefore
\begin{align}
\label{dilab-prop-lichnerowicz-a-bar-a-proof-eq-1}
\begin{split}
& \big[\overline{\partial}^E_N + \overline{\partial}^{E,*}_N,-\frac{\epsilon\sqrt{-1}}{4}(d_M\omega)(e_i,e_j)c(e_i)c(e_j)\big] \\
= \; & \frac{\epsilon\sqrt{-1}}{4}\big[c(e_k)\nabla^{\Lambda^\cdot(\overline{T^*N})\otimes E}_{e_k},(d_M\omega)(e_i,e_j)c(e_i)c(e_j)\big] \\
= \; &  \frac{\epsilon\sqrt{-1}}{4}\Big( \n^{\Lambda^\cdot(\overline{T^*N})\otimes E}_{e_i}(d_M\omega)(e_i,e_j)c(e_j) + (d_M\omega)(e_i,e_j)c(e_j)\nabla^{\Lambda^\cdot(\overline{T^*N})\otimes E}_{e_i}\Big) \;.
\end{split}
\end{align}
By \eqref{dilab-eq-point-lichnerowicz}, \eqref{dilab-eq-omega-big-e-explicite} and \eqref{dilab-prop-lichnerowicz-a-bar-a-proof-eq-1}, we get
\begin{align}
\label{dilab-eq-point-a-bar-a}
\begin{split}
& -\big(\overline{\partial}^E_N + \overline{\partial}^{E,*}_N\big)^2 
- da\,\frac{1}{2}\big(\overline{\partial}^E_N+\overline{\partial}^{E,*}_N\big) 
- d\bar{a}\,\big[\overline{\partial}^E_N+\overline{\partial}^{E,*}_N,\frac{\epsilon}{2}\omega^\mathscr{E}\big] 
+ dad\bar{a}\,\frac{\epsilon}{2}\omega^\mathscr{E} \\
= \; & \frac{1}{2}\Big( \n^{\Lambda^\cdot\overline{T^*N}\otimes E}_{e_i} 
- da\,\frac{1}{2}c(e_i)
- d\bar{a}\epsilon\,\frac{\sqrt{-1}}{2}(d_M\omega)(e_i,e_j)c(e_j) 
\Big)^2  \\
& - d\bar{a}\epsilon\,\big[ \n^{\Lambda^\cdot(\overline{T^*N})\otimes E}_{e_i} , \frac{1}{2}\omega^E - \frac{1}{8}(d_M\omega^{TN})(e_j,Je_j) \big] c(e_i) \\
& + dad\bar{a}\epsilon\,\big(\frac{1}{2}\omega^E - \frac{1}{8}(d_M\omega)(e_j,Je_j)\big) \\
& - \frac{1}{2}\big(R^E+\frac{1}{2}\tr[R^{TN}]\big)(e_i,e_j)c(e_i)c(e_j) - \frac{1}{8}r^N \;.
\end{split}
\end{align}

Comparing \eqref{dilab-eq-family-lichnerowicz},
\eqref{dilab-eq-point-lichnerowicz},   
\eqref{dilab-eq-point-a-bar-a}
with \eqref{dilab-eq-lichnerowicz-a-bar-a}, 
it only remains to show that
\begin{align}
\label{dilab-prop-lichnerowicz-a-bar-a-proof-eq-2}
\begin{split}
\sum_{i\neq j}^{}\langle S^{T_\mathbb{R}N}(e_i)e_j,f_\alpha \rangle f^\alpha c(e_i) 
c(e_j) = \; & 0 \;,\\
\sum_i\sum_{j\neq k}(d_M\omega)(e_i,e_j) \langle S^{T_\mathbb{R}N}(e_i)e_k,f_\alpha \rangle f^\alpha  c(e_j) c(e_k) = \; & 0 \;.
\end{split}
\end{align}

By \cite[\textsection 1(c)]{b}, if $U,V\in T\mathcal{N}$, 
then $S^{T_{\mathbb R}N}(U)V-S^{T_{\mathbb R}N}(V)U\in T_{\mathbb R}N$. 
Thus
\begin{equation}
\label{eq:err1}
\langle S^{T_\mathbb{R}N}(e_i)e_j,f_\alpha \rangle = 
\langle S^{T_\mathbb{R}N}(e_j)e_i,f_\alpha \rangle \;.
\end{equation}
By \eqref{eq:err1}, we get
the first identity in 
(\ref{dilab-prop-lichnerowicz-a-bar-a-proof-eq-2}).

Now we  prove the second identity in \eqref{dilab-prop-lichnerowicz-a-bar-a-proof-eq-2}. 
For simplicity, we introduce the following notation
\begin{equation}
\nabla_{f_\alpha} = i_{f_\alpha} d_M \;.
\end{equation}
By \cite[(1.5)]{b97}, we have
\begin{equation}
\langle S^{T_\mathbb{R}N}(e_i)e_k,f_\alpha \rangle =-\frac{1}{2} 
\left\langle \big(g^{T_\mathbb{R}N}\big)^{-1}\nabla_{f_{\alpha}}g^{T_\mathbb{R}N}(e_i)\,,\,e_k \right\rangle = 
-\frac{1}{2}\left(\nabla_{f_{\alpha}}\omega\right) (e_i,Je_k) \;.
\end{equation}
Therefore the second identity in \eqref{dilab-prop-lichnerowicz-a-bar-a-proof-eq-2} is equivalent to the follows:
\begin{equation}
\label{dilab-prop-lichnerowicz-a-bar-a-proof-eq-4}
\sum_i\sum_{j\neq k}(\nabla_{f_{\alpha}}\omega) (e_i,e_j) 
(\nabla_{f_{\beta}}\omega)(e_i,Je_k)  f^\alpha f^\beta c(e_j) c(e_k) = 0 \;.
\end{equation}
Since $(Je_i)_{1\leqslant i\leqslant n}$ is also an orthogonal basis of $T_\mathbb{R}N$, using the fact that $\omega$ and $d_M\omega$ are $J$-invariant, we get

\begin{align}
\label{dilab-prop-lichnerowicz-a-bar-a-proof-eq-5}
\begin{split}
& \sum_i\sum_{j\neq k}\left(\n_{f_\alpha}\omega\right) (e_i,e_j) 
\left(\n_{f_\beta}\omega\right) (e_i,Je_k)  f^\alpha f^\beta c(e_j) c(e_k) \\
= \; & \frac{1}{2} \sum_i\sum_{j\neq k}\left(\n_{f_\alpha}\omega\right) (e_i,e_j) 
\left(\n_{f_\beta}\omega\right) (e_i,Je_k)  f^\alpha f^\beta c(e_j) c(e_k) \\
& + \frac{1}{2} \sum_i\sum_{j\neq k}\left(\n_{f_\alpha}\omega\right) (Je_i,e_j) 
\left(\n_{f_\beta}\omega\right)(Je_i,Je_k)  f^\alpha f^\beta c(e_j) c(e_k) \\
= \; & \frac{1}{2} \sum_i\sum_{j\neq k}\left(\n_{f_\alpha}\omega\right) (e_i,e_j) 
\left(\n_{f_\beta}\omega\right) (e_i,Je_k)  f^\alpha f^\beta c(e_j) c(e_k) \\
& - \frac{1}{2} \sum_i\sum_{j\neq k}\left(\n_{f_\alpha}\omega\right) (e_i,Je_j) 
\left(\n_{f_\beta}\omega\right) (e_i,e_k)  f^\alpha f^\beta c(e_j) c(e_k) \;.
\end{split}
\end{align}
Exchanging the roles of $j,k$ and of $\alpha,\beta$, we obtain 
\begin{align}
\label{dilab-prop-lichnerowicz-a-bar-a-proof-eq-6}
\begin{split}
& \sum_i\sum_{j\neq k}\left(\n_{f_\alpha}\omega\right) (e_i,Je_j) 
\left(\n_{f_\beta}\omega\right) (e_i,e_k)  f^\alpha f^\beta c(e_j) c(e_k) \\
= \; &  \sum_i\sum_{j\neq k}\left(\n_{f_\beta}\omega\right) (e_i,Je_k) 
\left(\n_{f_\alpha}\omega\right) (e_i,e_j)  f^\beta f^\alpha c(e_k) c(e_j) \\
= \; &  \sum_i\sum_{j\neq k} \left(\n_{f_\alpha}\omega\right) (e_i,e_j)
\left(\n_{f_\beta}\omega\right) (e_i,Je_k) f^\alpha f^\beta c(e_j) c(e_k) \;.
\end{split}
\end{align}
By \eqref{dilab-prop-lichnerowicz-a-bar-a-proof-eq-5} and 
\eqref{dilab-prop-lichnerowicz-a-bar-a-proof-eq-6}, 
we get
\eqref{dilab-prop-lichnerowicz-a-bar-a-proof-eq-4}.
\end{proof}

\begin{proof}[Proof of Theorem \ref{dilab-prop-large-small-time-convergence}] 
The proof of \eqref{dilab-eq-large-time-convergence} follows the same argument as \cite[Theorem 3.16]{bl}.

Now we prove the first formula in \eqref{dilab-eq-small-time-convergence}. 
By Lemma \ref{dilab-prop-a-bar-a-formula}, 
it is sufficient to establish the asymptotics of the following terms as $t\rightarrow 0$ : 
\begin{align}
\label{dilab-eq-a0-proof-prop-small-time-convergence}
\begin{split}
& \trs \Big[\exp\big(-C^{\mathscr{E},2}_t 
- da\,\frac{1}{2}\big(\overline{\partial}^E_N+t\overline{\partial}^{E,*}_N\big) \\
& \hspace{35mm} - d\bar{a}\,\big[\overline{\partial}^E_N+t\overline{\partial}^{E,*}_N,\frac{\epsilon}{2}\omega^\mathscr{E}\big] 
+ dad\bar{a}\,\frac{\epsilon}{2}\omega^\mathscr{E}\big) \Big]^{\epsilon dad\bar{a}} \;,\\
& d_M \trs \Big[\frac{1}{2}N^{\Lambda^\cdot(\overline{T^*N})}\exp\big( D^{\mathscr{E},2}_t \big)\Big] \;.
\end{split}
\end{align}

As $t\rightarrow 0$, 
we claim that we can use equation \eqref{dilab-eq-lichnerowicz-a-bar-a} 
exactly as in Bismut-K\"{o}hler \cite[Theorem 3.22]{bk}. 
The main difference is that in \cite{bk}, 
the space of variations of the metrics is $1$-dimensional, 
while here it is the full basis $M$. 
By proceeding as in this reference, 
we get
\begin{align}
\label{dilab-eq-8-proof-prop-small-time-convergence}
\begin{split}
& \sqrt{2\pi i}\varphi \trs \Big[\exp\big(-C^{\mathscr{E},2}_t 
- da\,\frac{1}{2}\big(\overline{\partial}^E_N+t\overline{\partial}^{E,*}_N\big) \\
& \hspace{35mm} - d\bar{a}\,\big[\overline{\partial}^E_N+t\overline{\partial}^{E,*}_N,\frac{\epsilon}{2}\omega^\mathscr{E}\big] 
+ dad\bar{a}\,\frac{\epsilon}{2}\omega^\mathscr{E}\big) \Big]^{\epsilon dad\bar{a}} \\
= \; & q_*\Big[ \widetilde{\mathrm{Td}}(TN, 
g^{TN})*\widetilde{\mathrm{ch}}(E, g^E) \Big]  +  \mathscr{O}(t) \;.
\end{split}
\end{align}
This gives the asymptotics of the first term in \eqref{dilab-eq-a0-proof-prop-small-time-convergence}.

We turn to study the second term in \eqref{dilab-eq-a0-proof-prop-small-time-convergence}.
As $t\rightarrow 0$, by the local families index theorem technique \cite{b}, we get
\begin{equation}
\label{dilab-eq-3-proof-prop-small-time-convergence}
\varphi \trs \Big[tN^{\Lambda^\cdot(\overline{T^*N})}\exp\big( D^{\mathscr{E},2}_t \big)\Big] 
= q_*\left[\frac{\omega}{2\pi}\text{\rm Td}(TN, \nabla^{TN})\text{\rm ch}(E, \nabla^E)\right] + \mathscr{O}(\sqrt{t}) \;.
\end{equation}
Furthermore, by \cite[Theorems 2.11, 2.16]{bgs2}, 
the asymptotics of
$\trs \Big[N^{\Lambda^\cdot(\overline{T^*N})}\exp\big( D^{\mathscr{E},2}_t \big)\Big]$ 
is given by a Laurent series. 
By \eqref{dilab-eq-3-proof-prop-small-time-convergence}, we get
\begin{equation}
\label{dilab-eq-a1-proof-prop-small-time-convergence}
\varphi \trs \Big[N^{\Lambda^\cdot(\overline{T^*N})}\exp\big( D^{\mathscr{E},2}_t \big)\Big] 
= C_{-1} t^{-1}   + C_0 + \mathscr{O}(t) \;,
\end{equation}
with
\begin{equation}
\label{dilab-eq-a2-proof-prop-small-time-convergence}
C_{-1} = q_*\left[\frac{\omega}{2\pi}\text{\rm Td}(TN, \nabla^{TN})\text{\rm ch}(E, \nabla^E)\right] \;.
\end{equation}
Let $C_{-1}^{(p)}$ (resp.  $C_0^{(p)}$) be the  component of degree $p$ of $C_{-1}$ (resp.  $C_0$). 
By Remark \ref{dilab-rem-small-time-convergence}, for $p>0$, $C_{-1}^{(p)}=0$. 
Then
\begin{equation}
\label{dilab-eq-4-proof-prop-small-time-convergence}
\big(1 + N^{\Lambda^\cdot(T^*M)} + t\frac{\partial}{\partial t} \big) 
\trs\Big[N^{\Lambda^\cdot(\overline{T^*N})}\exp\big( D^{\mathscr{E},2}_t \big)\Big] 
= \sum_p \Big( (p+1) C_0^{(p)} \Big) + \mathscr{O}(t) \;.
\end{equation}

Applying \eqref{dilab-eq-8-proof-prop-small-time-convergence} with $\mathscr{E}$ replaced by $\mathscr{E}_+$ (see the proof of Proposition \ref{dilab-prop-trans-alpha-beta})
and taking the $dt$ component, 
we get
\begin{align}
\label{dilab-eq-5-proof-prop-small-time-convergence}
\begin{split}
& \varphi \trs \Big[ \exp\big( -C^{\mathscr{E}_+,2} 
- da\,\frac{1}{2}\big( \overline{\partial}^E_N + t\overline{\partial}^{E,*}_N \big) \\
& \hspace{35mm} - d\bar{a}\,\big[\overline{\partial}^E_N + t\overline{\partial}^{E,*}_N , \frac{\epsilon t}{2}\omega^{\mathscr{E}_+}\big] 
+ dad\bar{a}\,\frac{\epsilon t}{2}\omega^{\mathscr{E}_+}\big) \Big]^{\epsilon dad\bar{a}dt} \\
= \; & - \frac{1}{2} q_*\Big[ \mathrm{Td}'(TN, \nabla^{TN})\mathrm{ch}(E, \nabla^E)\Big] + \mathscr{O}(t) \;.
\end{split}
\end{align}
By Theorem \ref{dilab-thm-gen-flat2trivialclass}, Lemma \ref{dilab-prop-a-bar-a-formula-second} and
\eqref{dilab-eq-5-proof-prop-small-time-convergence}, 
we have
\begin{equation}
\label{dilab-eq-6-proof-prop-small-time-convergence}
\big(1 + N^{\Lambda^\cdot(T^*M)} + t\frac{\partial}{\partial t} \big) 
\trs\Big[N^{\Lambda^\cdot(\overline{T^*N})}\exp\big( D^{\mathscr{E},2}_t \big)\Big] 
= \text{ closed form } + \mathscr{O}(t) \;.
\end{equation}
By \eqref{dilab-eq-4-proof-prop-small-time-convergence} and 
\eqref{dilab-eq-6-proof-prop-small-time-convergence}, 
we have
\begin{align}
\label{dilab-eq-7-proof-prop-small-time-convergence}
 d_M C_0 = 0 \;.
\end{align}

By \eqref{dilab-eq-a1-proof-prop-small-time-convergence}, 
\eqref{dilab-eq-a2-proof-prop-small-time-convergence} and
\eqref{dilab-eq-7-proof-prop-small-time-convergence}, 
as $t\rightarrow 0$, 
we have
\begin{align}
\label{dilab-eq-10-proof-prop-small-time-convergence}
\begin{split}
& \sqrt{2\pi i} \varphi \, d_M  \trs \Big[ N^{\Lambda^\cdot(\overline{T^*N})}\exp\big( D^{\mathscr{E},2}_t \big)\Big] \\
= \; & d_M \varphi \trs \Big[ N^{\Lambda^\cdot(\overline{T^*N})}\exp\big( D^{\mathscr{E},2}_t \big)\Big] \\
= & \; \frac{1}{t} d_M q_*\left[\frac{\omega}{2\pi}\text{\rm Td}(TN, \nabla^{TN})\text{\rm ch}(E, \nabla^E)\right] + \mathscr{O}(\sqrt{t}) \;.
\end{split}
\end{align}
This gives the asymptotics of the second term in \eqref{dilab-eq-a0-proof-prop-small-time-convergence}.

The first formula in \eqref{dilab-eq-small-time-convergence} follows from 
Lemma \ref{dilab-prop-a-bar-a-formula}, 
\eqref{dilab-eq-8-proof-prop-small-time-convergence} and
\eqref{dilab-eq-10-proof-prop-small-time-convergence}.

The second formula in \eqref{dilab-eq-small-time-convergence} may be proved as  a consequence of the first one
by applying the same technique as in the proof of Proposition \ref{dilab-prop-trans-alpha-beta}.
\end{proof}

\subsection{Analytic torsion forms}
\label{dilab-subsect-torsionform}

\

We choose $g_1,g_2\in\smooth(\R_+,\R)\index{g1@$g_1$}\index{g2@$g_2$}$ satisfying
\begin{equation}
\label{dilab-eq-condition-inf-g1g2}
g_1(t) = 1 + \mathscr{O}\big(t\big) \;,\hspace{5mm}
g_2(t) = 1 + \mathscr{O}\big(t^2\big) \;,\hspace{5mm}\text{as }t \rightarrow 0 \;,
\end{equation}
\begin{equation}
\label{dilab-eq-condition-zero-g1g2}
g_1(t) = \mathscr{O}\big(e^{-t}\big) \;,\hspace{5mm}
g_2(t) = \mathscr{O}\big(e^{-t}\big) \;,\hspace{5mm}\text{as }t \rightarrow +\infty \;,
\end{equation}
and
\begin{align}
\label{dilab-eq-condition-g1g2}
\begin{split}
\int_0^1 \frac{g_1(t) - 1}{t} dt + 
\int_1^{+\infty} \frac{g_1(t)}{t}  = \; & \Gamma'(1) - 2 \;,\\ 
\int_0^1 \frac{g_2(t) - 1}{t^2} dt + 
\int_1^{+\infty} \frac{g_2(t)}{t^2} = \; & 1 \;.
\end{split}
\end{align}
Using Mellin tranformation, \eqref{dilab-eq-condition-g1g2} can be reformulated as follows
\begin{align}
\label{dilab-eq-condition-g1g2-zeta}
\begin{split}
\Big( \frac{d}{ds} \frac{1}{\Gamma(s)}\int_0^{+\infty} t^{s-1} g_1(t) dt \Big)_{s=0} = \; & - 2 \;,\\
\Big( \frac{d}{ds} \frac{1}{\Gamma(s)}\int_0^{+\infty} t^{s-2} g_2(t) dt \Big)_{s=0} = \; & 0 \;.
\end{split}
\end{align}

\begin{defn}
\label{dilab-def-analytic-torsion-form}
The analytic torsion forms $\mathscr{T}(g^{TN},g^E)\in\Omega^\mathrm{even}(M)\index{tf@$\mathscr{T}(g^{TN},g^E)$}$ are defined by
\begin{align}
\begin{split}
\mathscr{T}(g^{TN},g^E) 
= \; & - \int_{0}^{+\infty} 
\Big\{ \beta_t + \frac{g_1(t)-1}{2} {\chi}'(N,E) - \frac{g_1(t)}{2} n \chi(N,E) \\
& \hspace{25mm} + \frac{g_1(t)}{2} q_*\Big[\mathrm{Td}'(TN, \n^{TN})\mathrm{ch}(E, \n^E)\Big]  \\
& \hspace{25mm} + \frac{g_2(t)}{2t}  q_*\Big[\frac{\omega}{2\pi}\mathrm{Td}(TN, \n^{TN})\mathrm{ch}(E, \nabla^E)\Big] 
\Big\} \frac{dt}{t} \;.
\end{split}
\end{align}
\end{defn}

By Theorem \ref{dilab-prop-large-small-time-convergence}, 
$\mathscr{T}(g^{TN},g^E)$ is well-defined. 
Here we remark that $\mathscr{T}(g^{TN},g^E)$ is independent of $g_1$ and $g_2$.

\begin{prop}
\label{dilab-prop-diff-torsion}
We have
\begin{align}
\label{dilab-eq-prop-diff-torsion}
\begin{split}
d_M \mathscr{T}(g^{TN},g^E) 
= \; & q_*\Big[\widetilde{\mathrm{Td}}(TN, g^{TN})*\widetilde{\mathrm{ch}}(E, g^E)\Big] \\
& - f(H^\cdot(N,E), \n^{H^\cdot(N,E)}, g^{H^\cdot(N,E)}) \;.
\end{split}
\end{align}
\end{prop}
\begin{proof}
By Theorem \ref{dilab-thm-gen-flat2trivialclass}, 
$q_*\Big[\mathrm{Td}'(TN, \n^{TN})\mathrm{ch}(E, \n^E)\Big]$ is a constant function on $M$. 
Then, by Proposition \ref{dilab-prop-trans-alpha-beta}, we get
\begin{align}
\label{dilab-eq1-pf-prop-diff-torsion}
\begin{split}
& d_M \mathscr{T}(g^{TN},g^E) \\
= \; & - \int_{0}^{+\infty} 
\Big\{ d_M \beta_t  + \frac{g_2(t)}{2t} d_Mq_*\Big[\frac{\omega}{2\pi}\mathrm{Td}(TN, \n^{TN})\mathrm{ch}(E, \nabla^E)\Big] 
\Big\} \frac{dt}{t} \\
= \; & - \int_{0}^{+\infty} 
\Big\{ \frac{\partial}{\partial t}\alpha_t  + \frac{g_2(t)}{2t^2} d_Mq_*\Big[\frac{\omega}{2\pi}\mathrm{Td}(TN, \n^{TN})\mathrm{ch}(E, \nabla^E)\Big] 
\Big\} dt \;.
\end{split}
\end{align}
By Theorem \ref{dilab-prop-large-small-time-convergence}, 
\eqref{dilab-eq-condition-g1g2} and 
\eqref{dilab-eq1-pf-prop-diff-torsion},
we get \eqref{dilab-eq-prop-diff-torsion}.
\end{proof}

Proceeding in the same way as in \cite[Theorem 3.16]{bl}, 
we get
\begin{equation}
\trs\Big[ N^{\Lambda^\cdot(\overline{T^*N})} \exp\big(-tD^{E,2}_N\big) \Big] = \chi'(N,E) + \mathscr{O}\big(t^{-1}\big) \;,
\hspace{5mm}\text{as } t\rightarrow +\infty \;.
\end{equation}
For $s\in\mathbb{C}$ with $\mathrm{Re}(s)>n$, 
we define
\begin{equation}
\label{dilab-eq-def-zeta}
\theta(s) 
= -\frac{1}{\Gamma(s)}\int_0^{+\infty} t^{s-1}
\left[ \trs \big[ N^{\Lambda^\cdot(\overline{T^*N})}\exp\big(-tD^{E,2}_N\big) \big] - {\chi}'(N,E) \right] dt 
\index{theta@$\theta(s)$}\;.
\end{equation}
By \cite{s}, 
the function $\theta(s)$ admits a meromorphic continuation to the whole complex plane, which is regular at $0\in\mathbb{C}$.

Let $\mathscr{T}^{[0]}(g^{TN},g^E)$ be the component of $\mathscr{T}(g^{TN},g^E)$ of degree zero.

\begin{prop}
\label{dilab-prop-torsion-degzero}
We have
\begin{equation}
\label{dilab-eq-prop-torsion-degzero}
\mathscr{T}^{[0]}(g^{TN},g^E) = \frac{1}{2}{\theta}'(0) \;. 
\end{equation}
\end{prop}
\begin{proof}
By \eqref{dilab-eq-lc-superconnection-rescaling} and \eqref{dilab-eq-def-alphatbetat}, we get
\begin{align}
\begin{split}
\label{dilab-eq1-pf-prop-torsion-degzero}
\beta_t^{[0]} = \; & \trs\Big[ \frac{N^{\Lambda^\cdot(\overline{T^*N})}}{2}\big(1-2tD^{E,2}_N\big) \exp\big(-tD^{E,2}_N\big) \Big] \\
= \; & \frac{1}{2}\Big(1+2t\frac{\partial}{\partial t}\Big)\trs\Big[ N^{\Lambda^\cdot(\overline{T^*N})} \exp\big(-tD^{E,2}_N\big) \Big] \;.
\end{split}
\end{align}

By \eqref{dilab-eq-a1-proof-prop-small-time-convergence}, 
there exist $a_0,a_1\in\C$ such that,
as $t\rightarrow 0$, 
\begin{equation}
\label{dilab-eq2-pf-prop-torsion-degzero}
\trs\Big[ N^{\Lambda^\cdot(\overline{T^*N})} \exp\big(-tD^{E,2}_N\big) \Big] 
= a_{-1}t^{-1} + a_0 + \mathscr{O}\big(\sqrt{t}\big) \;.
\end{equation}
By \eqref{dilab-eq-small-time-convergence}, 
\eqref{dilab-eq1-pf-prop-torsion-degzero} and 
\eqref{dilab-eq2-pf-prop-torsion-degzero}, 
we get
\begin{equation}
\label{dilab-eq3-pf-prop-torsion-degzero}
a_0 = - q_*\left[\mathrm{Td}'(TN, \nabla^{TN})\mathrm{ch}(E, \nabla^E)\right] + n \chi(N,E) \;.
\end{equation}
By \eqref{dilab-eq-def-zeta},
\eqref{dilab-eq2-pf-prop-torsion-degzero} and
\eqref{dilab-eq3-pf-prop-torsion-degzero}, 
we get
\begin{equation}
\label{dilab-eq5-pf-prop-torsion-degzero}
{\theta}(0) = q_*\left[\mathrm{Td}'(TN, \nabla^{TN})\mathrm{ch}(E, \nabla^E)\right] - n \chi(N,E) + {\chi}'(N,E) \;.
\end{equation}

By Definition \ref{dilab-def-analytic-torsion-form}, 
\eqref{dilab-eq-condition-g1g2-zeta}, 
\eqref{dilab-eq-def-zeta} and
\eqref{dilab-eq1-pf-prop-torsion-degzero}, 
we have
\begin{align}
\label{dilab-eq6-pf-prop-torsion-degzero}
\begin{split}
& \mathscr{T}^{[0]}(g^{TN},g^E) \\
= \; & - \int_{0}^{+\infty} 
\Big\{ \frac{1}{2}\Big(1+2t\frac{\partial}{\partial t}\Big)\trs\Big[ N^{\Lambda^\cdot(\overline{T^*N})} \exp\big(-tD^{E,2}_N\big) \Big]
 - \frac{1}{2} {\chi}'(N,E) \\
& \hspace{10mm} + \frac{g_1(t)}{2} \left( q_*\Big[\mathrm{Td}'(TN, \n^{TN})\mathrm{ch}(E, \n^E)\Big] - n\chi(N,E) + {\chi}'(N,E) \right) \\
& \hspace{10mm} + \frac{g_2(t)}{2t}  q_*\Big[\frac{\omega}{2\pi}\mathrm{Td}(TN, \n^{TN})\mathrm{ch}(E, \nabla^E)\Big] 
\Big\} \frac{dt}{t} \\
= \; & - \frac{1}{2} \frac{d}{ds} \Big|_{s=0} \frac{1}{\Gamma(s)}  \int_{0}^{+\infty} 
t^{s-1}\Big(1+2t\frac{\partial}{\partial t}\Big) \Big\{ \trs\Big[ N^{\Lambda^\cdot(\overline{T^*N})} \exp\big(-tD^{E,2}_N\big) \Big] \\
&\hspace{75mm} - {\chi}'(N,E) \Big\} dt \\
& - \frac{1}{2} \frac{d}{ds}\Big|_{s=0} \frac{1}{\Gamma(s)}\int_0^{+\infty} t^{s-1} g_1(t) dt  \,
\Big( q_*\Big[\mathrm{Td}'(TN, \n^{TN})\mathrm{ch}(E, \n^E)\Big] \\
&\hspace{75mm} - n\chi(N,E) + {\chi}'(N,E) \Big) \\
& - \frac{1}{2} \frac{d}{ds}\Big|_{s=0} \frac{1}{\Gamma(s)}\int_0^{+\infty} t^{s-2} g_2(t) dt  \,
q_*\Big[\frac{\omega}{2\pi}\mathrm{Td}(TN, \n^{TN})\mathrm{ch}(E, \nabla^E)\Big] \\
= \; & \frac{d}{ds}\Big|_{s=0} \frac{1-2s}{2}\theta(s) 
+ q_*\Big[\mathrm{Td}'(TN, \n^{TN})\mathrm{ch}(E, \n^E)\Big] -n\chi(N,E) + {\chi}'(N,E) \\
= \; & \frac{1}{2}{\theta}'(0) - \theta(0) + q_*\Big[\mathrm{Td}'(TN, \n^{TN})\mathrm{ch}(E, \n^E)\Big] -\chi(N,E) + {\chi}'(N,E) \;.
\end{split}
\end{align}

By \eqref{dilab-eq5-pf-prop-torsion-degzero} and \eqref{dilab-eq6-pf-prop-torsion-degzero}, 
we obtain \eqref{dilab-eq-prop-torsion-degzero}.
\end{proof}



\begin{thebibliography}{BGS88b}

\bibitem[BeGV04]{bgv}
N.~Berline, E.~Getzler, and M.~Vergne, \emph{Heat kernels and {D}irac
  operators}, Grundlehren Text Editions, Springer-Verlag, Berlin, 2004,
  Corrected reprint of the 1992 original. \MR{2273508 (2007m:58033)}

\bibitem[BGS88a]{bgs2}
J.-M. Bismut, H.~Gillet, and C.~Soul{\'e}, \emph{Analytic torsion and
  holomorphic determinant bundles. {II}. {D}irect images and {B}ott-{C}hern
  forms}, Comm. Math. Phys. \textbf{115} (1988), no.~1, 79--126.

\bibitem[BGS88b]{bgs3}
\bysame, \emph{Analytic torsion and holomorphic determinant bundles. {III}.
  {Q}uillen metrics on holomorphic determinants}, Comm. Math. Phys.
  \textbf{115} (1988), no.~2, 301--351.

\bibitem[B86]{b}
J.-M. Bismut, \emph{The {A}tiyah-{S}inger index theorem for families of {D}irac
  operators: two heat equation proofs}, Invent. Math. \textbf{83} (1986),
  no.~1, 91--151.

\bibitem[B97]{b97}
\bysame, \emph{Holomorphic families of immersions and higher analytic torsion
  forms}, Ast\'erisque (1997), no.~244, viii+275.

\bibitem[BK92]{bk}
J.-M. Bismut and K.~K{\"o}hler, \emph{Higher analytic torsion forms for direct
  images and anomaly formulas}, J. Algebraic Geom. \textbf{1} (1992), no.~4,
  647--684.

\bibitem[BL95]{bl}
J.-M. Bismut and J.~Lott, \emph{Flat vector bundles, direct images and higher
  real analytic torsion}, J. Amer. Math. Soc. \textbf{8} (1995), no.~2,
  291--363.

\bibitem[BZ92]{bz}
J.-M. Bismut and W.~Zhang, \emph{An extension of a theorem by {C}heeger and
  {M}\"uller}, Ast\'erisque (1992), no.~205, 235, With an appendix by
  Fran{\c{c}}ois Laudenbach.

\bibitem[Ch79]{c-cm}
J.~Cheeger, \emph{Analytic torsion and the heat equation}, Ann. of Math. (2)
  \textbf{109} (1979), no.~2, 259--322.

\bibitem[M{\"u}78]{m-cm}
W.~M{\"u}ller, \emph{Analytic torsion and {$R$}-torsion of {R}iemannian
  manifolds}, Adv. in Math. \textbf{28} (1978), no.~3, 233--305.

\bibitem[M{\"u}93]{m-cm-2}
\bysame, \emph{Analytic torsion and {$R$}-torsion for unimodular
  representations}, J. Amer. Math. Soc. \textbf{6} (1993), no.~3, 721--753.

\bibitem[RS71]{rs}
D.~B. Ray and I.~M. Singer, \emph{{$R$}-torsion and the {L}aplacian on
  {R}iemannian manifolds}, Advances in Math. \textbf{7} (1971), 145--210.

\bibitem[RS73]{rs2}
\bysame, \emph{Analytic torsion for complex manifolds}, Ann. of Math. (2)
  \textbf{98} (1973), 154--177.

\bibitem[Se67]{s}
R.~T. Seeley, \emph{Complex powers of an elliptic operator}, Singular
  {I}ntegrals ({P}roc. {S}ympos. {P}ure {M}ath., {C}hicago, {I}ll., 1966),
  Amer. Math. Soc., Providence, R.I., 1967, pp.~288--307.

\bibitem[Zh16]{yzhang}
Y.~Zhang, \emph{A {R}iemann-{R}och-{G}rothendieck theorem for flat fibrations
  with complex fibers}, C. R. Math. Acad. Sci. Paris \textbf{354} (2016),
  no.~4, 401--406.

\end{thebibliography}

\providecommand{\bysame}{\leavevmode\hbox to3em{\hrulefill}\thinspace}
\providecommand{\MR}{\relax\ifhmode\unskip\space\fi MR }
\providecommand{\MRhref}[2]{%
  \href{http://www.ams.org/mathscinet-getitem?mr=#1}{#2}
}
\providecommand{\href}[2]{#2}

\end{document}